\documentclass[12pt]{article}
\usepackage{amsfonts}
\usepackage{amsthm}
\usepackage{amsmath}
\usepackage[all]{xy}
\usepackage{titlesec}
\usepackage{dsfont}
\usepackage{amssymb}
\usepackage{extarrows}
\usepackage{mathrsfs}
\usepackage{arydshln}
\usepackage{multirow}

\theoremstyle{theorem}
\newtheorem{theorem}{Theorem}[section]

\theoremstyle{definition}

\theoremstyle{proposition}
\newtheorem{proposition}[theorem]{Proposition}

\theoremstyle{corollary}
\newtheorem{corollary}[theorem]{Corollary}

\theoremstyle{lemma}
\newtheorem{lemma}[theorem]{Lemma}

\theoremstyle{lemma}
\newtheorem{remark}[theorem]{Remark}

\theoremstyle{remark}
\newtheorem*{note}{Note}

\theoremstyle{example}
\newtheorem{example}[theorem]{Example}

\renewcommand{\thesection}{\hspace{0.01pt}\arabic{section}}
\titleformat{\section}[hang]{\normalfont\normalsize\filright\bf}{\thesection}{0.5em}{}
\titlespacing{\section}{0em}{*0.5}{\wordsep}

\begin{document}
\title{\textbf{The decomposability of smash product of $\mathbf{A}_{n}^2$-complexes}}
\author{\Large Zhongjian  Zhu
\footnote{Email: \emph{zhuzhongjian@amss.ac.cn}}  \\
\normalsize \em{Institute of Mathematics, Academy of Mathematics and Systems Science}\\
\normalsize \em{No.55, Zhongguancun East Road, Haidian District, Beijing 100190, China }\\
\Large Jianzhong Pan
\footnote{Email: \emph{pjz@amss.ac.cn}}\\
\normalsize \em{Hua Loo-Keng Key Mathematical Laboratory}\\
\normalsize \em{Institute of Mathematics, Academy of Mathematics and Systems Science}\\
\normalsize \em{No.55, Zhongguancun East Road, Haidian District, Beijing 100190, China }}
\date{}

\maketitle
{\noindent\small{\bf{Abstract}}\quad
 In this paper, we determine the decomposability of smash product of two indecomposable
$\mathbf{A}_{n}^2$-complexes, i.e., $(n-1)$-connected finite CW-complexes with dimension at most $n+2$ $(n\geq 3)$.}

\

{\noindent\small{\bf{keywords}}\quad
 indecomposable; smash product; $\mathbf{A}_{n}^k$-complexes; cofbire sequence.}

 {\noindent\small{\bf{msc 2000}}\quad
55P10;55P15}

\newcommand{\lrg}[1]{\ensuremath{\langle #1 \rangle }}
\newcommand{\coH}[1]{\ensuremath{\mathbf{C}_{#1}}}
\newcommand{\CH}{\ensuremath{\mathbf{C}}}
\newcommand{\moo}[1]{\ensuremath{M_{2}^{#1}}}
\newcommand{\m}[2]{\ensuremath{M_{2^{#1}}^{#2}}}
\newcommand{\mo}[2]{\ensuremath{M(\mathbb{Z}/2^{#1},#2)}}
\newcommand{\s}[1]{\ensuremath{\scriptstyle{#1}}}
\newcommand{\z}[1]{\ensuremath{\mathbb{Z}/2^{#1}}}
\newcommand{\zz}{\ensuremath{\mathbb{Z}/2}}
\newcommand{\ch}[2]{\ensuremath{H^{#1}({#2}; \mathbb{Z}/2)}}
\newcommand{\CC}{\ensuremath{C^{5}_{u}\wedge C^{5,s}_{r}}}
\newcommand{\CCC}{\ensuremath{C^{5,s}_{r}\wedge C^{5,s'}_{r'}}}
\newcommand{\mm}[1]{\ensuremath{M_{2^{#1}}^{'3}}}
\newcommand{\A}[2]{\ensuremath{\mathbf{A}_{#1}^{#2}}}
\newcommand{\ZH}{\ensuremath{\mathbb{Z}}}
\newcommand{\ccc}[1]{\ensuremath{\mathbf{c}_{#1}}}
\newcommand{\rrr}[1]{\ensuremath{\mathbf{r}_{#1}}}

\newcommand{\mat}[1]{\ensuremath{\footnotesize{\Big(\begin{array}{c}
 \eta \\
  2^{#1} \\
  \end{array}
  \Big)}}}
\newcommand{\matt}[2]{\ensuremath{\footnotesize{\Big(\begin{array}{c}
  #1 \\
  2^{#2} \\
  \end{array}
  \Big)}}}
\newcommand{\ma}[2]{\ensuremath{\footnotesize{\Big(
\begin{array}{c}
  2^{#1},\eta \\
   0~,2^{#2} \\
    \end{array}
   \Big)}}}
\newcommand{\mtwo}[2]{\ensuremath{\footnotesize{\Big(\begin{array}{c}
 #1 \\
 #2 \\
  \end{array}
  \Big)}}}
\newcommand{\mthree}[3]{\ensuremath{\footnotesize{\Big(\begin{array}{c}
 #1 \\
 #2 \\
 #3 \\
  \end{array}
  \Big)}}}
\newcommand{\mfour}[4]{\ensuremath{\footnotesize{\Big(
\begin{array}{c}
  #1,#2 \\
  #3,#4 \\
    \end{array}
   \Big)}}}

\section{Introduction}
\label{intro}
For a finite CW-complex $X$ with co-H-space structure, if $X\simeq X_1\vee X_2$ with non-contractible $X_1$
and $X_2$, then $X$ is called  decomposable; otherwise X is called indecomposable. One of the basic problems in homotopy theory is to classify indecomposable homotopy types. Although, it is impossible to find all indecomposable homotopy types, it is indeed possible to solve this problem in special situations.
Let $\mathbf{A}_{n}^k$ $(n\geq k+1)$, the homotopy category consisting of $(n-1)$-connected finite CW-complexes with dimension at most $n+k$, any complex in $\mathbf{A}_{n}^k$ is a suspension and thus a co-H-space;  $\mathbf{F}_{n}^k$,  the full subcategory of $\mathbf{A}_{n}^k$ consisting of complexes with torsion free homology groups; $\mathbf{F}_{n(2)}^k$, the full subcategories of $\mathbf{A}_{n}^k$ consisting of complexes with 2-torsion free homology groups;  $\mathbf{F}_{n(2,3)}^k$, the full subcategories of $\mathbf{A}_{n}^k$ consisting of  complexes with 2 and 3 torsion free homology groups.

In 1950, S.C. Chang classified the indecomposable homotopy types in $\mathbf{A}_{n}^2 (n\geq 3)$ \cite{RefChang}, that is
 \begin{itemize}
   \item [(i)] Spheres:~~ $S^{n}$, $S^{n+1}$, $S^{n+2}$;
   \item [(ii)] Elementary Moore spaces:~~ $M_{p^{r}}^n$ , $M_{p^{r}}^{n+1}$  where $p$ is a prime, $r\in \mathbb{Z}^+$ and $M_{p^r}^{k}$ denotes $M(\mathbb{Z}/p^{r},n)$;
   \item [(iii)] Elementary Chang complexes:~~ $C_{\eta}^{n+2}$, $C^{(n+2),s}$, $C_{r}^{(n+2)}$, $C_{r}^{(n+2),s}$ (See Section~\ref{subsec2.1}) where $r,s\in \mathbb{Z}^+$.

   where $\mathbb{Z}^+$ denotes the set of positive integers.
 \end{itemize}
 Classification of  indecomposable homotopy types of $\mathbf{A}_{n}^3 (n\geq 4)$ are given by Baues and Hennes in
 \cite{RefBH} and for $k\geq 4$, classification of indecomposable homotopy types of $\mathbf{A}_{n}^k (n\geq k+1)$ is wild
in the sense similar to that in representation theory of finite dimensional algebras. Classification of  indecomposable
homotopy types of $\mathbf{F}_{n}^k (n\geq k+1)$ for $k=4,5,6$ are given by Baues and Drozd in \cite{RefTF5cels}, \cite{RefTF6cels} and \cite{RefDrFP} and it is wild for $k\geq 7$. Based on these earlier results, in our previous papers \cite{PZ23} and \cite{PZ2}, we classify indecomposable homotopy types of $\mathbf{F}_{n(2,3)}^k (n\geq k+1)$ for $k\leq 6$ and $\mathbf{F}_{n(2)}^4 (n\geq 5)$.

As pointed out by Jie Wu \cite{RefQuestion}, starting from an explicit space $X$, one obtains more indecomposable
 spaces from the self-smash products of $X$ since self-smash products of co-H-spaces admit decompositions.
This motivates us to consider the decomposability of smash products of different indecomposable complexes.
There are only a few results for this problem. The decomposability of $M(\mathbb{Z}/p^{r} , n)\wedge M(\mathbb{Z}/p^{s} , n)$ is well known \cite{RefNJ}.
Jie Wu \cite{RefWJ} proved that $M(\mathbb{Z}/2 , n)\wedge C_{\eta}^{n+2}$  and $C_{\eta}^{n+2}\wedge C_{\eta}^{n+2}$ are indecomposable.
As a main result in this paper, we determine the decomposability of all remaining smash product of two indecomposable
complexes in  $\mathbf{A}_{n}^2 (n\geq 3)$ and give the decomposition whenever possible. Since the suspension functor $\Sigma: \mathbf{A}_{n}^2 \rightarrow \mathbf{A}_{n+1}^2$ is an equivalence for $n\geq 3$. It suffices to  deal with the case $n=3$.

  \begin{theorem}[\textbf{Main theorem}]
   \label{Main theorem}
       For $r,s,r',s',u\in\mathbb{Z}^+$,
    \begin{itemize}
      \item [$\bullet$] $C_{r}^5\wedge C^{5,s}$, $C_{r}^5\wedge C_{r'}^5$, $ C^{5,s}\wedge C^{5,s'}$, $C_{\eta}^{5}\wedge C^{5}_{r}$,
          $C_{\eta}^{5}\wedge C^{5,s}$, $C_{\eta}^{5}\wedge C^{5,s}_{r}$ are indecomposable;

      \item [$\bullet$] $M_{2^{u}}^3\wedge C_{\eta}^5$ is indecomposable;

      \item [$\bullet$] $M_{2^{u}}^3\wedge C_{r}^5$ is
                     \begin{itemize}
                       \item [$\diamond$] indecomposable for $u>r$;
                       \item [$\diamond$] homotopy equivalent to $ M_{2^{u}}^3\wedge C_{\eta}^{5} \vee M_{2^{u}}^7$ for $r\geq u$;
                     \end{itemize}

      \item [$\bullet$] $M_{2^{u}}^3\wedge C^{5,s}$ is
                     \begin{itemize}
                       \item [$\diamond$]  indecomposable for $u>s$;
                       \item [$\diamond$] homotopy equivalent to $ M_{2^{u}}^3\wedge C_{\eta}^{5} \vee M_{2^{u}}^7$ for $s\geq u$
                     \end{itemize}

      \item [$\bullet$]  $M_{2^{u}}^3\wedge C_{r}^{5,s}$ is homotopy equivalent to
                \begin{itemize}
                  \item [$\diamond$]  $C_{r}^{8,s}\vee C_{r}^{9,s}$ for  $u> r,s$;
                  \item [$\diamond$] $ M_{2^{u}}^3\wedge C_{r}^{5} \vee M_{2^{u}}^7$ for $r<u\leq s$;
                  \item [$\diamond$] $ M_{2^{u}}^3\wedge C^{5,s} \vee M_{2^{u}}^7$ for  $s<u\leq r$;
                  \item [$\diamond$] $ M_{2^{u}}^3\wedge C_{\eta}^{5} \vee M_{2^{u}}^7\vee M_{2^{u}}^7$ for  $u\leq r$ and  $u\leq s$;
                \end{itemize}

      \item [$\bullet$] $C_{u}^{5}\wedge C_{r}^{5,s}$ is
               \begin{itemize}
                 \item [$\diamond$] homotopy equivalent to $C_{r}^{9,s}\vee C_{\eta}^{5}\wedge C_{r}^{5,s}$ for $u\geq r$ and $u\geq s$;
                 \item [$\diamond$] homotopy equivalent to $C_{s}^{9,r}\vee C_{\eta}^{5}\wedge C_{s}^{5,s}$ for $u=s<r$;
                 \item [$\diamond$] indecomposable, otherwise;
               \end{itemize}
   \item [$\bullet$] $C^{5,u}\wedge C_{r}^{5,s}$ is
               \begin{itemize}
                 \item [$\diamond$] homotopy equivalent to $C_{r}^{9,s}\vee C_{\eta}^{5}\wedge C_{r}^{5,s}$ for $u\geq r$ and $u\geq s$;
                 \item [$\diamond$] homotopy equivalent to $C_{s}^{9,r}\vee C_{\eta}^{5}\wedge C_{r}^{5,r}$ for $u=r<s$;
                 \item [$\diamond$] indecomposable, otherwise;
               \end{itemize}

   \item [$\bullet$]  $C_{r}^{5,s}\wedge C_{r'}^{5,s'}$
       \begin{itemize}
       \item [$\diamond$] if $s\geq r,r',s'$

         $$C_{r}^{5,s}\wedge C_{r'}^{5,s'}\simeq \left\{
     \begin{array}{ll}
       C_{r}^{9,s}\vee C_{r}^{9,s}\vee C_{\eta}^{5}\wedge C_{r}^{5,s}, & \hbox{$s=s'=r'>r$;} \\
       C_{r}^{9,s}\vee C^{5}_{r'}\wedge C_{r}^{5,s}, & \hbox{$s=s'>r'>r$;} \\
      C_{r}^{9,s}\vee C^{5,s'}\wedge C_{r}^{5,s}, & \hbox{$s=r'>s'>r$;} \\
        C_{r'}^{9,s'}\vee C_{r'}^{9,s'}\vee C_{\eta}^{5}\wedge C_{r'}^{5,s'}, & \hbox{$s\geq r\geq r',s'$;} \\
        C_{r'}^{9,s'}\vee C_{s'}^{9,r'}\vee C_{\eta}^{5}\wedge C_{r}^{5,r}, & \hbox{$s\geq r'>s'=r$;} \\
        C_{r'}^{9,s'}\vee C^{5}_{r}\wedge C_{r'}^{5,s'}, & \hbox{ohterwise.}
     \end{array}
   \right.$$
       \item [$\diamond$] if $r\geq r',s'$ and $r>s$

       $$C_{r}^{5,s}\wedge C_{r'}^{5,s'}\simeq \left\{
     \begin{array}{ll}
       C_{r}^{9,s}\vee C_{r}^{9,s}\vee C_{\eta}^{5}\wedge C_{r}^{5,s}, & \hbox{$r=r'=s'>s$;} \\
       C_{r}^{9,s}\vee C^{5,s'}\wedge C_{r}^{5,s}, & \hbox{$r=r'>s'>s$;} \\
      C_{r}^{9,s}\vee C^{5}_{r'}\wedge C_{r}^{5,s}, & \hbox{$r=s'>r'>s$;} \\
        C_{r'}^{9,s'}\vee C_{r'}^{9,s'}\vee C_{\eta}^{5}\wedge C_{r'}^{5,s'}, & \hbox{$r> s\geq r',s'$;} \\
        C_{r'}^{9,s'}\vee C_{s'}^{9,r'}\vee C_{\eta}^{5}\wedge C_{s}^{5,s}, & \hbox{$r\geq s'>r'=s$;} \\
        C_{r'}^{9,s'}\vee C^{5,s}\wedge C_{r'}^{5,s'}, & \hbox{ohterwise.}
     \end{array}
   \right.$$
     \end{itemize}
\end{itemize}
 \end{theorem}
\begin{remark}\label{remark 1.2}
  For $M_{p^r}^{3}$, prime $p\neq 2$, it is easy to check that
 $M_{p^r}^{3}\wedge C_{\eta}^5\simeq  M_{p^r}^{6}\vee M_{p^r}^{8}$;
     $M_{p^r}^{3}\wedge C_{r'}^5\simeq  M_{p^r}^{8}$;
    $M_{p^r}^{3}\wedge C^{5,s}\simeq  M_{p^r}^{6}$;
     $M_{p^r}^{3}\wedge C_{r}^{5,s}\simeq \ast$;
    $M_{p^r}^{3}\wedge M_{q^r}^{3}\simeq \ast$ for prime $q\neq p$ ($\ast$ is the point space). We will not discuss these cases any more in the following.
  \end{remark}

It is known from Theorem 1.2 of \cite{RefPaulWJ} that for any $p$-local CW-complex, there is a functorial decomposition $\Omega\Sigma X\simeq A^{min}(X)\times \Omega(\bigvee_{n=2}^{\infty} Q_{n}^{max}(X))$ which is useful to calculate homotopy groups of $\Sigma X$. $Q_{n}^{max}(X)$ is a wedge summand of $\Sigma X^{(n)}$, where $X^{(n)}$ is a $n$-fold self-smash product of $X$. In order to determine the homotopy type of $Q_{n}^{max}(X)$, it is significant to decompose $X^{(n)}$ to a wedge of indecomposable spaces.
The decomposition of self-smash product is easy for $M_{p^r}^{n} (p>2)$ and
$\m{s}{n} (s>1)$. The decomposition of self-smash product is obtained by Jie Wu \cite{RefWJ} for $\moo{n}$ and $C_{\eta}^n$.
In a sequel we will study the decomposition of self-smash product for Chang complexes.

 \textbf{Main Method} (Assume $C_1$ and $C_2$ are indecomposable homotopy types in $\mathbf{A}_{3}^2$)
\begin{itemize}
  \item [] The indecomposability of $C_1\wedge C_2$  is obtained by contradiction. Assuming that $C_1\wedge C_2$ is decomposable one gets a contradiction  by computing its homotopy invariants such as homotopy groups, cohomotopy groups or Steenrod operations.
  \item [] There are two ways to get the wedge decomposition of $C_1\wedge C_2$:

\textbf{ One way:} Applying Lemma \ref{lemma 2.3} to  cofibre sequence  $X\xrightarrow{f}Y\rightarrow C_1$ to get
  $X\wedge C_2 \xrightarrow{f\wedge 1}Y\wedge C_2\rightarrow C_1\wedge C_2$ which is also a cofibre sequence. Then rewrite $f\wedge 1\simeq (f_1, f_2,\cdots, f_t)$ under identification $X\wedge C_2\simeq X_1\vee X_2\vee \cdots \vee X_t$ or rewrite $f\wedge 1\simeq\footnotesize{\left(\begin{array}{c}
 f'_1 \\
f'_2 \\
 \cdots \\
  f'_{t'} \\
  \end{array}
  \right)}$ under identification $Y\wedge C_2\simeq Y_1\vee Y_2\vee \cdots \vee Y_{t'}$ and prove that  $f_i\simeq 0$ for some $i$ or $f'_{j}\simeq 0$ for some $j$ which will imply that  $\Sigma X_i$ or $Y_j$ is a wedge summand of $C_1\wedge C_{2}$.

\textbf{ Another way:} Firstly, observe that  $C_1\wedge C_2$ is a CW-complex with only one top cell $e^{10}$ and one bottom cell $S^6$; cancel the top cell and pinch the bottom cell to a point to get spaces $(C_{1}\wedge C_{2})^{(9)}$ and $(C_1\wedge C_2)/S^6$ respectively. The two spaces  have mapping cone structures  by Lemma \ref{lemma2.4};  Secondly, decompose
$(C_{1}\wedge C_{2})^{(9)}\simeq U_1\vee U_2\vee \cdots\vee U_{l}$ and $(C_1\wedge C_2)/S^6\simeq W_1\vee W_2\vee \cdots\vee W_{m}$ by matrix techniques introduced briefly in Subsection \ref{subsec2.1}. At last, from the decomposition of $(C_{1}\wedge C_{2})^{(9)}$, there is a cofibre sequence
$S^9\xrightarrow{f} U_1\vee U_2\vee \cdots\vee U_{l}\rightarrow C_1\wedge C_2$ and the map $f$ is determined by the decomposition of $(C_1\wedge C_2)/S^6$.
\end{itemize}

Section 2 contains necessary notations and lemmas. Related results of elementary Moore spaces and Chang-complexes are stated in Section 3. In
Section 5, we prove the last part of Theorem \ref{Main theorem} by determining the decomposition of $C_{r}^{5,s}\wedge C_{r'}^{5,s'}$ while the proof of other cases in Theorem \ref{Main theorem} is given in Section 4.

\section{Preliminaries}
\label{Prelimin}
\subsection{Some notations}
 \label{subsec2.1}
  \begin{itemize}
    \item All spaces are suspensions of simply connected finite  CW-complexes.

    \item $|G|$ denotes  the order of a group $G$ and  $|g|$ denotes the order of an element $g$ in group $G$. If $G$ is an abelian group with decomposition $G\cong C_1\oplus C_2\oplus\cdots \oplus C_m$, where $C_t$ is a cyclic group with order infinity or a power of a prime for $t=1,\cdots, m$, then define $dim~G:=m$.

    \item If $X$ is a subspace of $L$, $Y\simeq L/X$, then $i$ denotes the canonical inclusion $X\hookrightarrow L$, $q$ denotes the canonical projection
     $L\twoheadrightarrow Y$.
      Especially for Moore space $\m{k}{n}$, sometimes we denote $i: S^{n}\hookrightarrow \m{k}{n}$ by $i_{n}$ and $q: \m{k}{n}\twoheadrightarrow S^{n+1}$ by $q_{n}$.

    \item Denote by $H_{\ast}X:=H_{\ast}(X;\ZH)$ and $\ch{\ast}{X}$ the \textbf{reduced} homology groups and cohomology groups of space $X$ respectively.

    \item Let $\coH{f}$ be the mapping cone of a map $f$. Denote by $[\mathbf{C}_f,\mathbf{C}_{f'}]^{\alpha}_{\beta}$ the set of homotopy classes of maps $h$ which satisfy the the following homotopy commutative diagram:
$$\xymatrix{
   X\ar[r]^{f} \ar[d]^{\alpha}& Y\ar[r]\ar[d]^{\beta} & \coH{f}\ar[r] \ar[d]^{h}& \Sigma X\ar[d]^{\Sigma\alpha}\\
    X'\ar[r]^{f'} & Y'\ar[r] & \coH{f'}\ar[r] & \Sigma X'\\
       }$$

     \item For abelian groups $A_{i}$ and $B_{j}, (i=1,\cdots,t,  j=1,\cdots,s)$, we denote by $f:=(f_{ij})=\left(
                       \begin{array}{ccc}
                         f_{11} & \cdots & f_{1t} \\
                         \cdots &\cdots & \cdots \\
                         f_{s1}& \cdots & f_{st} \\
                       \end{array}
                     \right): \bigoplus_{i=1}^{t} A_{i}\rightarrow \bigoplus_{j=1}^{s} B_{j}$ a morphism such that $p_{B_{i}}fj_{A_{j}}=f_{ij}$, where $j_{A_{j}}$ and $p_{B_{i}}$ are canonical inclusions and projections respectively. Sometimes, $(f_{ij})$ is written graphically as follows to indicate the domain and codomain
 $$ \footnotesize{\begin{tabular}{c|cccc|}
      \multicolumn{1}{c} {}  & $A_{1}$ &  $A_{2}$& $\cdots$  & \multicolumn{1}{c} {$A_{t}$} \\
       \cline{2-5}
      $B_1$ & $f_{11}$ & $f_{12}$ & $\cdots$ & $f_{1t}$ \\
      $B_2$ & $f_{21}$ & $f_{22}$ & $\cdots$ & $f_{2t}$ \\
      $\vdots$ & $\vdots$ & $\vdots$ & $\vdots$ & $\vdots$ \\
      $B_s$ & $f_{s1}$ & $f_{s2}$ & $\cdots$ & $f_{st}$ \\
      \cline{2-5}
    \end{tabular} }$$

     \item  For finite CW-complexes $X_{i}$ and $Y_{j}, (i=1,\cdots,t,  j=1,\cdots,s)$, let
            $$f:=(f_{ij})=\left(
                       \begin{array}{ccc}
                         f_{11} & \cdots & f_{1t} \\
                         \cdots &\cdots & \cdots \\
                         f_{s1}& \cdots & f_{st} \\
                       \end{array}
                     \right): \bigvee_{i=1}^{t} X_{i}\rightarrow \bigvee_{j=1}^{s} Y_{j}$$ be a map such that $p_{Y_{i}}fj_{X_{j}}\simeq f_{ij}$, where $j_{X_{j}}$ and $p_{Y_{i}}$ are canonical inclusions and projections respectively.  Similarly, $(f_{ij})$ is written as
 $$ \footnotesize{\begin{tabular}{c|cccc|}
      \multicolumn{1}{c} {}  & $X_{1}$ &  $X_{2}$& $\cdots$  & \multicolumn{1}{c} {$X_{t}$} \\
       \cline{2-5}
      $Y_1$ & $f_{11}$ & $f_{12}$ & $\cdots$ & $f_{1t}$ \\
      $Y_2$ & $f_{21}$ & $f_{22}$ & $\cdots$ & $f_{2t}$ \\
      $\vdots$ & $\vdots$ & $\vdots$ & $\vdots$ & $\vdots$ \\
      $Y_s$ & $f_{s1}$ & $f_{s2}$ & $\cdots$ & $f_{st}$ \\
      \cline{2-5}
    \end{tabular}} $$
Generally,  $(f_{ij})$ is not unique up to homotopy. However, if $s=1$ or $X_{i}$ $(i=1,\cdots,t)$ and $Y_{j}$ $(j=1,\cdots,s)$ are in the category $\mathbf{A}^{k}_{n}$ for some $k\geq n+2$, then $(f_{ij})$ is unique.

 Thus in category  $\mathbf{A}^{k}_{n} (k\geq n+2)$,
 $$[\bigvee_{i=1}^{t} X_{i}, \bigvee_{j=1}^{s} Y_{j}]\cong \{(f_{ij})| f_{ij}\in [X_{i},Y_{j}]\}$$
 The composition (sum) of the maps is compatible with the product (sum) of the matrices. Thus a self-map is homotopy equivalent if and only if the corresponding matrix is invertible. Moreover for two matrices
 $$(f_{ij}), (g_{ij}):  \bigvee_{i=1}^{t} X_{i}\rightarrow \bigvee_{j=1}^{s} Y_{j},$$
we call  $(f_{ij})\cong(g_{ij})$ if there are invertible matrices
$$(\alpha_{ij}): \bigvee_{i=1}^{t} X_{i}\rightarrow  \bigvee_{i=1}^{t} X_{i}, ~~~~~~(\beta_{ij}): \bigvee_{j=1}^{s} Y_{j}\rightarrow \bigvee_{j=1}^{s} Y_{j} $$ such that  $(\beta_{ij})(f_{ij})(\alpha_{ij}) \simeq (g_{ij})$. It is clear that if $(f_{ij})\cong(g_{ij})$, then the mapping cones  $\coH{(f_{ij})}$ and $\coH{(g_{ij})}$ are homotopy equivalent to each other.

  To simplify the text, we fix names for some elementary transformations as follows
            \begin{itemize}
              \item [(i)]  $-\rrr{n}$ ($-\ccc{n}$) :  composing $n$-th row (column) with $-id$;
              \item [(ii)]   $\ccc{m}f+\ccc{n}$:  adding the $m$-th column,  composed with map $f$,  to the $n$-th column;
              \item [(iii)] $g\rrr{m}+\rrr{n}$:  adding the $m$-th row,  composed with map $g$,  to the $n$-th row;
              \item [(iv)] $k\rrr{m}+\rrr{n}$ ($k\ccc{m}+\ccc{n}$): adding $k$ times of the $m$-th row (column) to the $n$-th row (column), where $k\in \ZH^{+}$;
            \end{itemize}
     \item In the following of the paper, $S_{\omega}^{k}=S^{k}$  for $\omega\in \{1,2,3,4,a,b\}$.
     \item $\eta=\eta_{n}\in [S^{n+1}, S^{n}]$ is a Hopf map for $n\geq 3$ and $\varrho =\varrho _{n} \in [S^{n+3}, S^{n}]\cong \ZH/24$ is a fixed generator for $n\geq 5$.

     \item $\kappa, \kappa', \varepsilon, \varepsilon'\in  \{0,1\}$.

     \item The $k$ times of the identity map $id$ of $\Sigma X$ is written as $k:\Sigma X\rightarrow \Sigma X$ for nonzero integer $k$ (hence, $1=id$);

     \item For $k\geq 5$, and  $r,s\in \mathbb{Z}^+$, let

     $C^{k,s}=(S^{k-2}\vee S^{k-1})\bigcup_{\mat{s}}\mathbf{C}S^{k-1}=S^{k-2}\bigcup_{\eta q}\mathbf{C}\m{s}{k-2}$;

     $C^{k}_{r}=S^{k-2}\bigcup_{(2^r,\eta)}\CH(S^{k-2}\vee S^{k-1})=\m{r}{k-2}\bigcup_{i\eta}\CH S^{k-1}$;

     $C^{k,s}_{r}=(S^{k-2}\vee S^{k-1})\bigcup_{\ma{r}{s}}\CH(S^{k-2}\vee S^{k-1})=(\m{r}{k-2}\vee S^{k-2})\bigcup_{\matt{i\eta}{s}}\mathbf{C}S^{k-1}$ $=S^{k-2}\bigcup_{(2^r, \eta q)}\CH (S^{k-2}\vee \m{s}{k-2})=\m{r}{k-1}\bigcup_{i\eta q}\CH \m{s}{k-1}$;

     $C^{k}_{\eta}=S^{k-2}\bigcup_{\eta}\mathbf{C}S^{k-1}$,

     \end{itemize}

  \subsection{Spanier-Whitehead duality}
 \label{subsec2.2}
   If $\A{n}{k}$ is in the stable range, i.e., $[X,Y]\xlongrightarrow{\Sigma^m}[\Sigma^m X,\Sigma^m Y]$ is isomorphic for any $X,Y\in\A{n}{k}$ and any $m\in \mathbb{Z}^{+}$, then there is a contravariant isomorphism of additive categories
   $$D=D_{2n+k}: \A{n}{k}\rightarrow \A{n}{k} $$
   which is called Spanier-Whitehead duality (or $(2n+k)$-duality).

   $D$ satisfies the following properties from \cite{RefHB2}, \cite{RefCo} and  \cite{RefRMS}:

   \begin{proposition}\label{proposition 2.1}
   \begin{itemize}
     \item []
     \item [$(i)$] $D^2$ is equal to the identity functor;
     \item [$(ii)$] $[X,Y]\xlongrightarrow{D (\cong)} [DY,DX]\cong [DY\wedge X, S^{2n+k}]$;
     \item [$(iii)$] $[S^{n+q}, DX]\cong [X, S^{n+k-q}]$ and $[S^{n+q}, X]\cong [DX, S^{n+k-q}]$for $n\leq q\leq n+k$;
     \item [$(iv)$] $D(X\vee Y)\simeq DX\vee DY$;
     \item [$(v)$] $D(X\wedge Y)\simeq DX\wedge DY$, that is for $X\in\A{n}{k}, Y\in \A{m}{l}$, then $X\wedge Y\in\A{n+m}{k+l}$ and
        $D_{2(n+m)+k+l}(X\wedge Y)\simeq D_{2n+k}X\wedge D_{2m+l}Y$.
     \item [$(vi)$] Let $\{X,Y\}:=\lim\limits_{m \rightarrow +\infty}[\Sigma^m X, \Sigma^m Y]$. Then for any CW-complex $Z$,
     $$\{X\wedge Y, Z\}\cong \{X\wedge S^{2n+k}, DY\wedge Z\}$$
   \end{itemize}
   \begin{note} It follows from $(i)$ and $(iv)$ above that $X$ is indecomposable if and only if $DX$ is indecomposable.
    \end{note}

 \end{proposition}

   \begin{example}(Page $49$ of \cite{RefJ.H.C})
    For the Spanier-Whitehead duality $D: \A{n}{2}\rightarrow \A{n}{2} (n\geq 3)$, we have
    $DS^n=S^{n+2}$, $DS^{n+1}=S^{n+1}$, $DM_{p^r}^n=M_{p^r}^{n+1}$, $DC^{n+2}_{\eta}=C^{n+2}_{\eta}$, $DC_{r}^{(n+2),s}=C_{s}^{(n+2),r}$, $DC_{r}^{n+2}=C^{(n+2),r}$.
   \end{example}

  \subsection{Some lemmas}
 \label{subsec2.3}

\begin{lemma}\label{lemma 2.3}( \cite{RefCo})
 For a  cofibre  sequence  $X\xlongrightarrow{f} Y\xlongrightarrow{i} \coH f$,
 $$X\wedge Z\xlongrightarrow{f\wedge id}Y\wedge Z\xlongrightarrow{i\wedge id}\coH f\wedge Z$$
 is also a cofibre sequence. That is $\coH{f\wedge id}\simeq \coH {f}\wedge Z$.
\end{lemma}

\begin{lemma}\label{lemma2.4}(Lemma 14.30. of \cite{RefRMS})
 For $$X\xlongrightarrow{f} U\xlongrightarrow{i} \coH f, ~~~ Y\xlongrightarrow{g} V\xlongrightarrow{i} \coH g$$
 $\coH f\wedge \coH g=(U\wedge V)\bigcup_{\mu} \CH (X\wedge V\vee U\wedge Y)\bigcup_{\nu} \CH \CH (X\wedge Y)$ where
    $\mu=(f\wedge id, id\wedge g)$  and $(\coH f\wedge \coH g)/(U\wedge V) \simeq (\Sigma(X\wedge V)\vee \Sigma(U\wedge Y))\bigcup_{\nu'}\CH \Sigma(X\wedge Y)$, where
    $\nu'=\left(
            \begin{array}{c}
              \Sigma id\wedge g \\
              -\Sigma f\wedge id \\
            \end{array}
          \right)$.
\end{lemma}

\begin{lemma}\label{lemma2.5}(Lemma 6.2.1 of \cite{RefAlgMethod})
  Let  $$X\xlongrightarrow{f} Y\xlongrightarrow{i} \coH f\xlongrightarrow{q}\Sigma X$$ be a cofibre sequence. If $f$ is null homotopic, then
   \begin{itemize}
      \item [(i)] there is a retraction $r: \coH {f}\rightarrow Y$ of $i$, such that $ri\simeq id$ and $\coH {f}\rightarrow \Sigma X \vee \coH {f}\xlongrightarrow{1\vee r} \Sigma X\vee Y $ is a homotopy equivalence, where the first map is the standard coaction.
     \item [(ii)] there is a section $\tau: \Sigma X\rightarrow \coH f $ such that $q\tau\simeq id$ and $\Sigma X\vee Y \xlongrightarrow{(\tau, i)} \coH f$ is a homotopy equivalence.
   \end{itemize}
\end{lemma}

\begin{lemma}\label{lemma2.6}
  Let $A\in \A{6}{4}$, with homology groups of the form $\mathbb{Z}^r\oplus \mathbb{Z}/2^{r_1}\oplus \cdots \oplus \mathbb{Z}/2^{r_\varsigma}$ for some nonnegative integers $r,r_1,\cdots,r_\varsigma$. Suppose that
  \begin{itemize}
    \item [(i)]   $dim~H_{9}A+dim~H_{10}A=1$  and $H^{6}(A; \zz)\cong H^{10}(A; \zz)\cong \zz$  with  generators $a_{6}$ and $a_{10}$ respectively,  satisfying $Sq^4a_{6}=a_{10}$;
    \item [(ii)]  $dim~H^{8}(A; \zz)\geq 2$ and there are nonzero elements $a_{8} \neq a'_{8}\in H^{8}(A; \zz)$ such that $Sq^2 a_{8}=Sq^2 a'_{8}=a_{10}$;
    \item [(iii)] Moreover $Sq^2 a_{6}=a_{8}+a'_{8}+a''_{8}\neq 0$ for some $a''_{8}\in H^{8}(A; \zz)$ such that $Sq^2 a''_{8}=0$.
  \end{itemize}
 If $A\simeq X\vee Y$ and $H_{6}X\neq 0$, then  $H_{t}X\cong H_{t}A$ for $t=6,9,10$ and  $dim~H^{8}(X; \zz)\geq 2$, hence $dim~H_{7}X+ dim~H_{8}X\geq 2$.
\end{lemma}
   \begin{proof}
   Let
$$\xymatrix{X \ar@<0.5ex>[r]^-{j_{1}} & A\ar@<0.5ex>[l]^-{p_{1}} \ar@<-0.5ex>[r]_-{p_{2}}& Y \ar@<-0.5ex>[l]_-{j_{2}}\\}$$
where  $j_{1}, j_{2}, p_{1}, p_{2}$ be the canonical inclusions and projections.
$$\xymatrix{\ch{\ast}{X} \ar@<0.5ex>[r]^-{p_{1}^{*}} & \ch{\ast}{A}\ar@<0.5ex>[l]^-{j_{1}^*} \ar@<-0.5ex>[r]_-{j_{2}^*}& \ch{\ast}{Y} \ar@<-0.5ex>[l]_-{p_{2}^*}\\},$$
where $j_{u}^*p_{u}^*=id$ which implies that $p_{u}^*$ is injective and $j_{u}^*$ is surjective for $u=1,2$.

 Since $H_{6}X\neq 0$,  we get that $H_{6}X\cong H_{6}A$, $H_{6}Y=0$ and $H^{6}(p_1;\ZH/2)$ is isomorphic, hence there is $0\neq x_6\in \ch{6}{X}$ such that $p_{1}^*(x_6)=a_6$.
 It follows from  $p_{1}^{*}(Sq^{4}x_{6})=Sq^{4}p_{1}^{*}(x_{6})=a_{10}\neq 0$ that $0\neq Sq^4x_{6}\in \ch{10}{X}$ which implies $H_9X\cong H_9A$ and $H_{10}X\cong H_{10}A$ by (i).
 By $j_{1}^*(a_{10})=j_{1}^*(p_{1}^{*}(Sq^{4}x_{6}))= Sq^4x_6\neq 0$ and $Sq^2j_1^*(a_{8})=Sq^2j_1^*(a'_{8})=j_{1}^*(a_{10})$,
 we get $j_{1}^*(a_{8})\neq 0$ and $j_{1}^*(a'_{8})\neq 0$ in $ \ch{8}{X}$.

Since $Sq^2 a_{6}=a_{8}+a'_{8}+a''_{8}\neq 0$ and $Sq^2 a''_{8}=0$,
$p_{1}^*(Sq^2x_6)=Sq^2(a_{6})=a_{8}+a'_{8}+a''_{8} \neq 0$ and $p_1^{*}Sq^2Sq^2x_6=2a_{10}=0$, thus $Sq^2x_6\neq 0$ and $Sq^2Sq^2x_6=0$. But
  $Sq^2(j_1^*(a_{8}))=j_1^*(a_{10})\neq 0$, we have $j_1^{*}(a_{8})\neq Sq^2x_6$, thus $dim~\ch{8}{X}\geq 2$.

 It follows from $\ch{8}{X}=Hom(H_8X,\zz)\oplus Ext(H_7X,\zz)$ that  $dim~H_{7}X+ dim~H_{8}X\geq 2$.
   \end{proof}

 A complex $X$ is called 2-local if all homotopy groups or equivalently all homology groups of $X$ are finitely generated $\mathbb{Z}_{(2)}$-module, where $\mathbb{Z}_{(2)}$ is 2-localization of $\mathbb{Z}$. Let $X_{(2)}$ be the 2-localization of $X$ and
  denote  by $X\simeq_{(2)}Y$ if $X_{(2)}\simeq Y_{(2)}$.

\begin{lemma}\label{lemma2.7}
\begin{itemize}
  \item []
  \item [(i)] Let $X_{1}=S^{m}\cup_{f_1}\CH X'_{1}$, $X_{2}=S^{m}\cup_{f_2}\CH X'_{2}$ be two (resp. $2$-local) complexes, where $X'_{1}$ and $X'_{2}$ are $m$-connected. If $X_{1}\simeq X_{2}$ and $H_{m}X_1=\z{s}$ for some $s\in \mathbb{Z}^{+}$, then  $X_{1}/S^m\simeq_{(2)} X_{2}/S^m$ (resp. $X_{1}/S^m\simeq X_{2}/S^m$), i.e.,
      $\Sigma X'_{1}\simeq_{(2)}  \Sigma X'_{2}$ (resp. $\Sigma X'_{1}\simeq \Sigma X'_{2}$ ).

\item [(ii)] Let $X_{1}=X_{1}^{(n-1)}\cup_{g_1}\CH S^{n-1}$, $X_{2}=X_{2}^{(n-1)}\cup_{g_2}\CH S^{n-1}$ be two (resp. $2$-local) complexes, where $X_{1}^{(n-1)}$ and $X_{2}^{(n-1)}$ are $(n-1)$-skeleton of $X_{1}$ and $X_{2}$ respectively. If $X_{1}\simeq X_{2}$ and $[X_1, S^n]=\z{t}$ for some $t\in \mathbb{Z}^{+}$, then $X_{1}^{(n-1)}\simeq_{(2)} X_{2}^{(n-1)}$ (resp. $X_{1}^{(n-1)}\simeq X_{2}^{(n-1)}$).
\end{itemize}
\end{lemma}

\begin{proof}
It suffices to prove when $X_{1}$ and $X_{2}$ are 2-local.

The proof of (i):

 There are cofibre sequences $$X'_{j}\xlongrightarrow{f_{j}}S^{m}\xlongrightarrow{i_{j}}X_{j}\xlongrightarrow{q_{j}}\Sigma X'_{j}\xlongrightarrow{\Sigma f_{j}} S^{m+1}~~~~j=1,2.$$
 Given a homotopy equivalence $\alpha : X_{1}\xlongrightarrow{\simeq}X_{2}$, it induces
   $$\alpha_{\ast} : \pi_{m}X_{1}\xlongrightarrow{\cong } \pi_{m}X_{2}\cong \z{s}.$$
 Since $i_j$ is a generator of $\pi_{m}X_{j}$ for $j=1,2$, there is an odd integer $k$ such that $\alpha_{\ast}ki_{1}=i_{2}\in \pi_{m}X_{2}$.
 Since $X_{1}$ and $X_{2}$ are 2-local, $k\alpha : X_1\rightarrow X_{2}$ is also a homotopy equivalence. Thus there is a commutative diagram
  $$\xymatrix{
   S^{m}\ar[r]^{i_{1}} \ar@{=}[d]& X_{1}\ar[r]^{q_{1}}\ar[d]^{k\alpha (\simeq)} & \Sigma X'_{1}\ar[r] \ar@{.>}[d]^{\alpha'}& S^{m+1}\ar@{=}[d]\\
    S^{m}\ar[r]^{i_{2}} & X_{2}\ar[r]^{q_{2}} &  \Sigma X'_{2}\ar[r] & S^{m+1}\\
       }$$
  and  $$  \alpha':  \Sigma X'_{1}\xlongrightarrow{\simeq} \Sigma X'_{2}.$$

 The proof of (ii) is dual to the proof of (i) by investigating the cofibre sequences
  $$S^{n-1}\xlongrightarrow{g_{j}}X_{j}^{(n-1)}\xlongrightarrow{i_{j}}X_{j}\xlongrightarrow{q_{j}}S^{n} ~~j=1,2.$$ and using isomorphism
  $[X_{1}, S^{n}]\cong [X_{2}, S^{n}]$.
  \end{proof}

 \section{Moore spaces and Chang-complexes}
 \label{sec3}

   In this section, we will collect some basic facts about Chang-complexes and Moore spaces.

  \subsection{ Moore spaces}
 \label{sec3.1}
  Firstly, from Proposition 3E.3 of \cite{RefHatcher},  the Steenrod square action on $\m{r}{n}$ is given by

     $$Sq^1: \ch{n}{\m{r}{n}}\rightarrow \ch{n+1}{\m{r}{n}}~ \text{is}~\left\{
                                                                         \begin{array}{ll}
                                                                           isomorphic , & \hbox{$r=1$;} \\
                                                                           0, & \hbox{$r>1$.}
                                                                         \end{array}
                                                                       \right.$$
 Secondly, we list some results of maps between Moore spaces from \cite{RefBH}.

  $$[\m{r}{n},\m{t}{n}]=\left\{
                        \begin{array}{ll}
                          \ZH/4\lrg {B(\chi)}, & \hbox{$r=t=1$;} \\
                          \ZH/2^{min(r,t)} \lrg {B(\chi)} \oplus \ZH/2 \langle i\eta q \rangle, & \hbox{ohterwise.}
                        \end{array}
                      \right.~~~~~~~~(n\geq 3)$$
 where $B(\chi)$ is given by Proposition (2.3) of \cite{RefBH}, which satisfies
 $$H_{n}B(\chi)=\chi: \z{r}\rightarrow\z{t},~~~~~~  \chi(1)=1.$$

  $B(\chi)\in  [\m{r}{n},\m{t}{n}]_{id}^{2^{r-t}}$ for $r\geq t$ and $B(\chi)\in  [\m{r}{n},\m{t}{n}]^{id}_{2^{r-t}}$ for $r\leq t$;
  If $r=t$, then $B(\chi)=id$ and if $r=t=1$, then $i\eta q= 2B(\chi)=2.$

$$[S^{n+1}, \m{t}{n}]=\zz \lrg{i\eta},~~~~[\m{t}{n}, S^{n}]=\zz\lrg{\eta q}~~~~~~~~(n\geq 3)$$

 $$[\m{t}{n+1}, S^{n}]=\left\{
                              \begin{array}{ll}
                              \ZH/4\lrg{\eta^{1}}, & \hbox{$t=1$;} \\
                              \ZH/2\lrg{\eta^{t}}\oplus\ZH/2\lrg{\eta\eta q}, & \hbox{$t>1$.}
                              \end{array}
                      \right.~~~~~~~~(n\geq 3)$$

  $$[S^{n+2}, \m{t}{n}]=\left\{
                                  \begin{array}{ll}
                                   \ZH/4\lrg{\xi_{1}}, & \hbox{$t=1$;} \\
                                    \ZH/2\lrg{\xi_{t}}\oplus\ZH/2\lrg{i\eta\eta}, & \hbox{$t>1$.}
                                 \end{array}
                          \right.~~~~~~~~(n\geq 4)$$

    Here we choose a generator $\xi_{1}$ and set $\xi_{t}=B(\chi)\xi_{1}$. The generator $\eta^{1}$ is the dual map of $\xi_{1}$ and $\eta^{t}=\eta^{1}B(\chi)$. And $q\xi_{t}=\eta$, $\eta^{t}i=\eta$ for $t\geq 1$.

 $$[\m{s}{n+1}, \m{r}{n}]=\left\{
                              \begin{array}{ll}
                              \ZH/2\lrg{\xi_{1}^{1}}\oplus  \ZH/2\lrg{\eta_{1}^{1}}, & \hbox{$s=r=1$;} \\
                              \ZH/4\lrg{\xi_{1}^{s}}\oplus  \ZH/2\lrg{\eta_{1}^{s}}, & \hbox{$s>1=r$;} \\
                              \ZH/2\lrg{\xi_{r}^{1}}\oplus  \ZH/4\lrg{\eta_{r}^{1}}, & \hbox{$s=1<r$;} \\
                               \ZH/2\lrg{\xi_{r}^{s}}\oplus  \ZH/2\lrg{\eta_{r}^{s}}\oplus \ZH/2\lrg{i\eta\eta q}, & \hbox{otherwise,} \\
                              \end{array}
                      \right.~~~~~~~~(n\geq 4)$$
    where $\xi_{r}^{s}=B(\chi)\xi_{1}\in [\m{s}{n+1}, \m{r}{n}]^{\eta}_{0}$, $\eta^s_r=i\eta^{1}B(\chi)\in [\m{s}{n+1}, \m{r}{n}]_{\eta}^{0}$.
Note that  $$2\xi_{1}^{s}=i\eta\eta q ~(s>1);~~   2\eta_{r}^{1}=i\eta\eta q ~(r>1).$$

   Let $\lambda_{11}:=\xi_{r}^{s}+\eta_{r}^{s}$, then $$[\m{s}{n+1}, \m{r}{n}]^{\eta}_{\eta}=\{\lambda_{11}, \lambda_{11}+i\eta\eta q\}.$$

   $[S^{n+3}, \m{r}{n}]$ is given by the following Lemma
  \begin{lemma}\label{lemma3.1}
   Let $n\geq 5$. Then
   $$[S^{n+3}, \moo{n}]=\ZH/2\lrg{i\varrho}\oplus \ZH/2\lrg{\rho_{1}}$$
   $$[S^{n+3}, \m{r}{n}]=\ZH/4\lrg{i\varrho}\oplus \ZH/2\lrg{\rho_{r}}  ~~(r>1)$$
 where $\rho_{r}$ is some element of  $[S^{n+3}, \m{r}{n}]$ such that $q\rho_{r}=\eta\eta$ for $r\geq 1$.
\end{lemma}
 \begin{proof}
  $[S^{n+3}, \moo{n}]$ is obtained from Lemma 5.2 and Theorem 5.11 of \cite{RefWJ}.
  For $r>1$ there are two exact sequences
    $$[S^{k},S^{n}]\xlongrightarrow{(2^{r})^{\ast}}[S^{k},S^{n}]\rightarrow [S^{k},\m{r}{n}]\xlongrightarrow{q^{\ast}} [S^{k},S^{n+1}]\xlongrightarrow{(2^{r})^{\ast}}[S^{k},S^{n+1}], ~~k=n+2,n+3.$$
 The following commutative diagram is induced by $\eta: S^{n+3}\rightarrow S^{n+2}$
   $$\xymatrix{
   0\ar[r]& \ZH/2\ar[r]\ar[d] & [S^{n+2},\m{r}{n}]\ar[r]^{~~~~q^{\ast}} \ar[d]^{\eta^{\ast}}& \ZH/2\ar@/^/[l]^{~~~~\sigma} \ar[r]\ar@{=}[d]& 0 \\
    0\ar[r] & \ZH/4 \ar[r] &[S^{n+3},\m{r}{n}]\ar[r]^{~~~~q^{\ast}} & \ZH/2 \ar[r]& 0 \\
       }$$
 It is known that the upper exact sequence splits. If $\sigma$ is the section of $q^{\ast}$, then $\eta^{\ast}\sigma$ is the section of the  lower $q^{\ast}$.
\end{proof}

 At last, it follows from Lemma \ref{lemma2.5} that
  \begin{corollary}\label{corollary3.2}
  For $s\geq r$ and $s>1$
 $$\xymatrix{
  \m{r}{n}\wedge S^{m}\ar[r]^{1\wedge i_{m}}& \m{r}{n}\wedge\m{s}{m}\ar[r]^{1\wedge q_{m}}\ar@/^/[l]^{\tau_{m}}& \m{r}{n}\wedge S^{m+1}\ar@/^/[l]^{\sigma_{m}}  \\
       }$$
and
   $$\xymatrix{
   S^{m}\wedge\m{r}{n}\ar[r]^{i_{m}\wedge 1 }&\m{s}{m}\wedge \m{r}{n}\ar[r]^{q_{m}\wedge 1 }\ar@/^/[l]^{\tau'_{m}}&S^{m+1} \wedge \m{r}{n} \ar@/^/[l]^{\sigma'_{m}}  \\
       }$$

  where $\tau_{m}(1\wedge i_{m})=1$, $(1\wedge q_{m})\sigma_{m}=1$ and $\tau'_{m}(i_{m}\wedge 1)=1$, $( q_{m} \wedge 1)\sigma'_{m}=1$.

  \end{corollary}

\subsection{Chang-complexes}
 \label{subsec3.2}

   Firstly, since $C_{1}^{k,1}$ and $\Sigma^{k-4} \moo{1}\wedge\moo{1}$ are both indecomposable $\mathbf{A}_{k-2}^{2}$-complexes with the same homology groups, thus
   $$C_{1}^{k,1}\simeq\Sigma^{k-4} \moo{1}\wedge\moo{1}~~(k\geq 5).$$

 \textbf{Cofibre sequences for Chang-complexes }

For $C\in \{C^{k}_{r}, C^{k,s}, C_{r}^{k,s}~|~k\geq 5, r,s\in \mathbb{Z}^{+}\}$,  it can be written as mapping cones of different maps, that is, there are different cofibre sequences for $C$.

 \begin{itemize}
   \item  The cofibre sequence for $C^{k}_{\eta}$
          \begin{itemize}
        \item [\textbf{Cof} :] $S^{k-1}\xlongrightarrow{\eta}S^{k-2}\xlongrightarrow{i_{\eta}} C_{\eta}^k\xrightarrow{q_{\eta}} S^{k}\rightarrow S^{k-1}$
        \end{itemize}
   \item  The cofibre sequences for $C^{k}_{r}$
      \begin{itemize}
        \item [\textbf{Cof1} :] $S^{k-2}\vee S^{k-1}\xlongrightarrow{(2^r, \eta)}S^{k-2}\xlongrightarrow{i_{S}}C^{k}_{r}\xlongrightarrow{q_{S}}S^{k-1}\vee S^{k}\rightarrow S^{k-1}$;
        \item [\textbf{Cof2} :] $S^{k-1}\xlongrightarrow{i\eta}\m{r}{k-2}\xlongrightarrow{i_{M}}C^{k}_{r}\xlongrightarrow{q_{M}}S^{k}\rightarrow \m{r}{k-1}$;
        \item [\textbf{Cof3} :] $S^{k-2}\xlongrightarrow{i_{\eta}2^r}C_{\eta}^{k}\xlongrightarrow{i_{C}}C^{k}_{r}\xlongrightarrow{q_{C}}S^{k-1}\rightarrow C_{\eta}^{k+1}$;
      \end{itemize}

   \item  The cofibre sequences for $C^{k,s}$
       \begin{itemize}
        \item [\textbf{Cof1} :] $S^{k-1}\xlongrightarrow{\mat{s}}S^{k-2}\vee S^{k-1}\xlongrightarrow{i_{S}}C^{k,s}\xlongrightarrow{q_{S}} S^{k}\rightarrow S^{k-1}\vee S^k$;
        \item [\textbf{Cof2} :] $\m{s}{k-2}\xlongrightarrow{\eta q}S^{k-2}\xlongrightarrow{i_{M}}C^{k,s}\xlongrightarrow{q_{M}}\m{s}{k-1}\rightarrow S^{k-1}$;
        \item [\textbf{Cof3} :] $C_{\eta}^{k-1}\xlongrightarrow{2^sq_{\eta}}S^{k-1}\xlongrightarrow{i_{C}}C^{k,s}\xlongrightarrow{q_{C}}C_{\eta}^{k}\rightarrow S^{k}$;
      \end{itemize}

   \item  The cofibre sequences for $ C_{r}^{k,s}$
        \begin{itemize}
        \item [\textbf{Cof1} :] $S^{k-2}\vee S^{k-1}\xlongrightarrow{\ma{r}{s}}S^{k-2}\vee S^{k-1}\xlongrightarrow{i_{S}}C^{k,s}_{r}\xlongrightarrow{q_{S}}S^{k-1}\vee S^{k}\rightarrow S^{k-1}\vee S^{k}$;
        \item [\textbf{Cof2} :] $\m{s}{k-2}\xlongrightarrow{i\eta q}\m{r}{k-2}\xlongrightarrow{i_{M}}C^{k,s}_{r}\xlongrightarrow{q_{M}}\m{s}{k-1}\rightarrow \m{r}{k-1}$;
        \item [\textbf{Cof3} :] $S^{k-2}\vee \m{s}{k-2}\xlongrightarrow{(2^r, \eta q)}S^{k-2}\xlongrightarrow{i_{\overline{M}}}C^{k,s}_{r}\xlongrightarrow{q_{\overline{M}}}S^{k-1}\vee\m{s}{k-1}\rightarrow S^{k-1}$;
        \item [\textbf{Cof4} :] $S^{k-1}\xlongrightarrow{\matt{i\eta}{s}}\m{r}{k-2}\vee S^{k-1}\xlongrightarrow{i_{\underline{M}}}C^{k,s}_{r}\xlongrightarrow{q_{\underline{M}}}S^{k}\rightarrow \m{r}{k-1}\vee S^{k}$;
        \item [\textbf{Cof5} :] $C_{r}^{k-1}\xlongrightarrow{2^sp_{1}q_{S}}S^{k-1}\xlongrightarrow{i_{\underline{C}}}C^{k,s}_{r}\xlongrightarrow{q_{\underline{C}}}C_{r}^{k}\rightarrow S^{k}$, where $2^sp_{1}q_{S}$ is the composition of $C_{r}^{k-1}\xlongrightarrow{q_{S}}S^{k-1}\vee S^{k-2}\xlongrightarrow{p_{1}}S^{k-1}\xlongrightarrow{2^{s}}S^{k-1}$;
        \item [\textbf{Cof6} :] $S^{k-2}\xlongrightarrow{i_{S}j_{1}2^r}C^{k,s}\xlongrightarrow{i_{\overline{C}}}C^{k,s}_{r}\xlongrightarrow{q_{\overline{C}}}S^{k-1}\rightarrow C^{(k+1),s}$, where $i_{S}j_{1}2^r$ is the composition of $S^{k-2}\xlongrightarrow{2^{r}}S^{k-2}\xlongrightarrow{j_{1}}S^{k-2}\vee S^{k-1}\xlongrightarrow{i_{S}}C^{k,s}$.
      \end{itemize}
 \end{itemize}

  \textbf{Homologies and Cohomologies }

  $H_{\ast}C^{k,s}=\left\{
                     \begin{array}{ll}
                       \mathbb{Z}, & \hbox{$\ast=k-2$;} \\
                       \z s, & \hbox{$\ast=k-1$;} \\
                       0, & \hbox{otherwise.}
                     \end{array}
                   \right.$;~~~~ $H_{\ast}C^{k}_{r}=\left\{
                     \begin{array}{ll}
                       \z r, & \hbox{$\ast=k-2$;} \\
                       \mathbb{Z}, & \hbox{$\ast=k$;} \\
                       0, & \hbox{otherwise.}
                     \end{array}
                   \right.$

 $H_{\ast}C^{k,s}_{r}=\left\{
                     \begin{array}{ll}
                       \mathbb{Z}/2^r, & \hbox{$\ast=k-2$;} \\
                       \z s, & \hbox{$\ast=k-1$;} \\
                       0, & \hbox{otherwise.}
                     \end{array}
                   \right.$;~~~~  $H_{\ast}C^{k}_{\eta}=\left\{
                     \begin{array}{ll}
                        \mathbb{Z}, & \hbox{$\ast=k-2$;} \\
                       \mathbb{Z}, & \hbox{$\ast=k$;} \\
                       0, & \hbox{otherwise.}
                     \end{array}
                   \right.$

$H^{\ast}(C^{k}_r; \zz )=H^{\ast}(C^{k,s}; \zz )=\left\{
                     \begin{array}{ll}
                       \mathbb{Z}/2, & \hbox{$\ast=k-2, k-1, k$;} \\
                       0, & \hbox{otherwise.}
                     \end{array}
                   \right.$

    $H^{\ast}(C^{k,s}_{r}; \zz)=\left\{
                     \begin{array}{ll}
                       \mathbb{Z}/2, & \hbox{$\ast=k-2,k$;} \\
                       \zz\oplus \zz , & \hbox{$\ast=k-1$;} \\
                       0, & \hbox{otherwise.}
                     \end{array}
                   \right.$;~~$H^{\ast}(C^{k}_{\eta}; \zz)=\left\{
                     \begin{array}{ll}
                        \mathbb{Z}/2, & \hbox{$\ast=k-2$;} \\
                       \mathbb{Z}/2, & \hbox{$\ast=k$;} \\
                       0, & \hbox{otherwise.}
                     \end{array}
                   \right.$

$Sq^2: H^{k-2}(C; \zz)\rightarrow H^{k}(C; \zz)$ is isomorphic for $C=C^{k,s}, C^{k}_{r}, C^{k,s}_{r}, C^{k}_{\eta}.$

 $Sq^1=0:~ H^{k-2}(C^{k,s}; \zz)\rightarrow H^{k-1}(C^{k,s}; \zz)$

 $Sq^1: H^{k-1}(C^{k,s}; \zz)\rightarrow H^{k}(C^{k,s}; \zz)$ is $\left\{
                                                                 \begin{array}{ll}
                                                                   0, & \hbox{$s>1$;} \\
                                                                   isomorphic , & \hbox{$s=1$.}
                                                                 \end{array}
                                                               \right.$

$Sq^1: H^{k-2}(C^{k}_r; \zz)\rightarrow H^{k-1}(C^{k}_r; \zz)$ is $\left\{
                                                                 \begin{array}{ll}
                                                                   0, & \hbox{$r>1$;} \\
                                                                   isomorphic , & \hbox{$r=1$.}
                                                                 \end{array}
                                                               \right.$

 $Sq^1=0: H^{k-1}(C^{k}_{r}; \zz)\rightarrow H^{k}(C^{k}_r; \zz)$.

  $Sq^1$ on $ H^{\ast}(C^{k,s}_r; \zz)$ is given by the following Lemma
  \begin{lemma}\label{lemma3.3}
  Let $v_{k-2}, v_{k}$  be generators of $ H^{k-2}(C^{k,s}_r; \zz)$ and  $ H^{k}(C^{k,s}_r; \zz)$ respectively. Then there are generators $v_{k-1},\overline{ v}_{k-1}$ of $ H^{k-1}(C^{k,s}_r; \zz)$ such that
\begin{align}Sq^1v_{k-2}=\left\{
                \begin{array}{ll}
                  v_{k-1}, & \hbox{$r=1$;} \\
                  0, & \hbox{$r>1$.}
                \end{array}
              \right.;~~~~ Sq^1\overline{v}_{k-1}=\left\{
                \begin{array}{ll}
                  v_{k}, & \hbox{$s=1$;} \\
                  0, & \hbox{$s>1$.}
                \end{array}
              \right.;~~~Sq^1v_{k-1}=0.                              \label{Sq1}
\end{align}
  \end{lemma}
\begin{proof}
   To simplify the notation, take $k=5$.
   Using  \textbf{Cof2} and \textbf{Cof3} of $C_{r}^{5,s}$,we have a split exact sequence
   $$0\rightarrow\ch{4}{\m{s}{4}}\xlongrightarrow{q_{M}^{\ast}}\ch{4}{C_{r}^{5,s}}\xlongrightarrow{i_{M}^{\ast}}\ch{4}{\m{r}{3}}\rightarrow 0$$
     and an isomorphism
    $$\ch{4}{S^4}\oplus \ch{4}{\m{s}{4}}\xlongrightarrow{(p_{1}^{\ast}, p_{2}^{\ast})~\cong}\ch{4}{S^4\vee \m{s}{4}}\xlongrightarrow{q_{\overline{M}}^{\ast}~\cong}\ch{4}{C_{r}^{5,s}}$$ where $p_{l}$ is the canonical projection for $l=1,2.$
   Let $u_{S}^{4}, u_{M}^{4}$ be the generators of $\ch{4}{S^{4}}$ and $\ch{4}{\m{s}{4}}$ respectively.

  Note that $p_{2}q_{\overline{M}}=q_{M}$.
  Define $$\overline{v}_{4}:=q_{\overline{M}}^{\ast}p_{2}^{\ast}(u_{M}^{4})=q_{M}^{\ast}(u_{M}^{4}), ~~~~ v_{4}:=\left\{
                                                                                           \begin{array}{ll}
                                                                                             Sq^1v_{3}, & \hbox{$r=1$;} \\
                                                                                             q_{\overline{M}}^{\ast}p_{1}^{\ast}(u_{S}^{4}), & \hbox{$r>1$.}
                                                                                           \end{array}
                                                                                         \right.$$

 Clearly, $v_4\neq \overline{v}_{4}$ for $r>1$. For $r=1$, since $\ch{3}{C_{1}^{5,s}}\xlongrightarrow{i_{M}^{\ast}~(\cong)}\ch{3}{M_{2}^{3}}$, $i_{M}^{\ast}v_{4}=$ $i_{M}^{\ast}Sq^{1}v_{3}=$ $Sq^{1}i_{M}^{\ast}v_{3}\neq 0.$ But  $i_{M}^{\ast}\overline{v}_{4}=0$, we also get $v_{4}\neq \overline{v}_{4}$.

   $$Sq^1\overline{v}_{4}=Sq^1q_{M}^{\ast}u_{M}^{4}=q_{M}^{\ast}Sq^1u_{M}^{4} \left\{
                \begin{array}{ll}
                  v_{5}, & \hbox{$s=1$;} \\
                  0, & \hbox{$s>1$.}
                \end{array}
              \right.$$
$$Sq^1{v}_{4}= \left\{
                \begin{array}{ll}
                  Sq^1 q_{\overline{M}}^{\ast}p_{1}^{\ast}u_{S}^{4}= q_{\overline{M}}^{\ast}p_{1}^{\ast}Sq^1u_{S}^{4}=0, & \hbox{$r>1$;} \\
                  Sq^1Sq^1v_3=0, & \hbox{$r=1$.}
                \end{array}
              \right.$$

   Applying \textbf{Cof4} of $C^{5,s}_{r}$  and by the commutative diagram
   $$\xymatrix{
  \ch{3}{C^{5,s}_{r}}  \ar[d]_{Sq^1} \ar@^{(->}[r]^{i_{\underline{M}}^{\ast}~~~}
                & \ch{3}{\m{r}{3}\vee S^4}  \ar[d]^{Sq^1}  \\
  \ch{4}{C^{5,s}_{r}}  \ar[r]_{i_{\underline{M}}^{\ast} ~~~}
                &  \ch{4}{\m{r}{3}\vee S^4}  }$$
  we get  $$Sq^{1}v_{3}=\left\{
                           \begin{array}{ll}
                              v_{4}\neq 0, & \hbox{$r=1$;} \\
                             0, & \hbox{$r>1$.}
                           \end{array}
                         \right.$$
\end{proof}

 \textbf{homotopy groups and cohomotopy groups}

For $k\geq 5$

\begin{itemize}
  \item [] $\pi_{k-1}C^{k}_{r}=0$;
  \item [] $\pi_{k}C^{k}_{r}=\ZH/2\lrg{(j_{1}\eta)_S^-}\oplus\ZH \lrg{(2j_{2})_S^-}$ where  $(j_{1}\eta)_S^-=q_{S\ast}^{-1}(j_{1}\eta)$ and $(2j_{2})_S^-=q_{S\ast}^{-1}(2j_{2})$;
  \item [] $\pi_{k-1}C^{k,s}=\z{s+1}\lrg{i_{S}j_{2}}$ with $i_{S}j_{1}\eta=2^{s}i_{S}j_{2}$;
  \item [] $\pi_{k}C^{k,s}=\ZH/2\lrg{i_{S}j_{2}\eta}$;
  \item [] $\pi_{k-1}C^{k,s}_{r}=\z{s+1}\lrg{i_{\underline{M}}j_{2}}$ with $i_{\underline{M}}j_{1}i\eta=2^{s}i_{\underline{M}}j_{2}$ or
    \newline
        $\pi_{k-1}C^{k,s}_{r}=\z{s+1}\lrg{i_{S}j_{2}}$ with $i_{S}j_{1}i\eta=2^{s}i_{S}j_{2}$;
  \item []$\pi_{k}C^{k,s}_{r}=\ZH/2\lrg{i_{\underline{M}}j_{1}\xi_{r}}\oplus\ZH/2\lrg{i_{\underline{M}}j_{2}\eta}$;
\end{itemize}

\begin{itemize}
  \item [] $[C^{k}_{r}, S^{k-2}]=\ZH/2\lrg{\eta p_{1}q_{S}}$;
  \item [] $[C^{k}_{r},S^{k-1}]=\z{r+1}\lrg{p_{1}q_{S}}$ with $\eta p_{2}q_{S}=2^rp_{1}q_{S}$;
  \item [] $[C^{k,s}, S^{k-2}]=\ZH\lrg{(2p_{1})_S^{-}}\oplus \ZH/2\lrg{(\eta p_{2})_S^-}$ where $(2p_{1})_S^{-}=(i_{S}^{\ast})^{-1}(2p_{1})$ and $(\eta p_{2})_S^-=(i_{S}^{\ast})^{-1}(\eta p_{2})$;
  \item [] $[C^{k,s}, S^{k-1}]=0$;
  \item [] $[C^{k,s}_{r}, S^{k-2}]=\ZH/2 \lrg{(\eta qp_{1})_M^-}\oplus \ZH/2\lrg{(\eta p_{2})_M^{-}}$ where $(\eta qp_{1})_M^-=(i_{\underline{M}}^{\ast})^{-1}(\eta qp_{1})$ and $(\eta p_{2})_M^{-}=(i_{\underline{M}}^{\ast})^{-1}(\eta p_{2})$;
  \item [] $[C^{k,s}_{r}, S^{k-1}]=\z{r+1}\lrg{p_{1}q_{S}}$ with $\eta p_{2}q_{S}=2^{r}p_{1}q_{S}$ or
  \newline
$[C^{k,s}_{r}, S^{k-1}]=\z{r+1}\lrg{p_{1}q_{\overline{M}}}$ with $\eta q p_{2}q_{\overline{M}}=2^{r}p_{1}q_{\overline{M}}$;
  \item [] $[C^{k,s}_{r}, S^{k}]=\z{s}\lrg{qp_{2}q_{\underline{M}}}=\z{s}\lrg{p_{1}q_{S}}$.

\end{itemize}
 where $X_{u}\xlongrightarrow{j_{u}}X_{1}\vee X_{2}$ is the canonical inclusion and  $X_{1}\vee X_{2}\xlongrightarrow{p_{u}}X_{u}$ is the canonical projections for $u=1,2$.

\begin{remark}\label{remark 3.4}
$[X,Y]$ for $X,Y$ being indecomposable homotopy types of $\mathbf{A}_{n}^2$ are given by Part IV of \cite{RefJ.H.C}. We will use these results directly in the following of this paper.
  \end{remark}

 \section{Determination of the decomposability  except that of $C_{r}^{5,s}\wedge C_{r'}^{5,s'}$}
   \label{sec4}

\subsection{Some indecomposable cases}
   \label{subsec4.1}

\begin{theorem}\label{theorem4.1}
   $\m{u}{3}\wedge C_{\eta}^5$, $C_{r}^5\wedge C^{5,s}$,$C_{r}^5\wedge C_{r'}^5$, $C^{5,s}\wedge C^{5,s'}$, $C_{\eta}^5\wedge C_{r}^5$, $C_{\eta}^5\wedge C^{5,s}$ and $C_{\eta}^5\wedge C_{r}^{5,s}$ are indecomposable for any $u, r,r',s,s'\in \ZH^{+}$.
\end{theorem}

    \begin{proof}
The proof of all cases in the Theorem \ref{theorem4.1} are similar. We give a proof only for the case $C_{r}^5\wedge C_{r'}^5$ .

 Let  $u_{k}, u'_{k}$ be generators of $\ch{k}{C_{r}^{5}}$ and  $\ch{k}{C_{r'}^{5}}$ respectively for $k=3,4,5$.
 Let $m_{uv}$ be the minimum of non-negative integers $u$ and $v$.
$$H_{\ast}(C_{r}^5\wedge C_{r'}^{5})=\left\{
                                                \begin{array}{ll}
                                                  \z{m_{r,r'}}, & \hbox{$\ast=6$} \\
                                                   \z{m_{r,r'}}, & \hbox{$\ast=7$} \\
                                                  \z{r}\oplus\z{r'}, & \hbox{$\ast=8$} \\
                                                  0, & \hbox{$\ast=9$} \\
                                                  \ZH, & \hbox{$\ast=10$} \\
                                                  0, & \hbox{otherwise}
                                                \end{array}
                                              \right.$$

          $$\ch{\ast}{C_{r}^5\wedge C_{r'}^{5}}=\left\{
                                                \begin{array}{ll}
                                                 \zz\{u_{3}\otimes u'_{3}\}, & \hbox{$\ast=6$} \\
                                                 \zz\{u_{3}\otimes u'_{4}, u_{4}\otimes u'_{3}\}, & \hbox{$\ast=7$} \\
                                                 \zz\{u_{3}\otimes u'_{5}, u_{4}\otimes u'_{4}, u_{5}\otimes u'_{3}\}, & \hbox{$\ast=8$} \\
                                                 \zz\{u_{4}\otimes u'_{5}, u_{5}\otimes u'_{4}\}, & \hbox{$\ast=9$} \\
                                                 \zz\{u_{5}\otimes u'_{5}\}, & \hbox{$\ast=10$} \\
                                                  0, & \hbox{otherwise}
                                                \end{array}
                                              \right.$$

     By the  Steenrod operation action on Chang-complexes given in Section \ref{subsec3.2}, we get

    (i) $Sq^4(u_{3}\otimes u'_{3})=u_{5}\otimes u'_{5}$; (ii)$Sq^2(u_{3}\otimes u'_{5})=Sq^2(u_{5}\otimes u'_{3})=u_{5}\otimes u'_{5}$;

    (iii) $Sq^2(u_{3}\otimes u'_{3})=\left\{
                                      \begin{array}{ll}
                                        u_{3}\otimes u'_{5}+u_{5}\otimes u'_{3}, & \hbox{otherwise} \\
                                        u_{3}\otimes u'_{5}+u_{4}\otimes u'_{4}+u_{5}\otimes u'_{3}, & \hbox{$r=r'=1$}
                                      \end{array}
                                    \right.$, $Sq^2(u_{4}\otimes u'_{4})=0$;

    (iv) $Sq^2(u_{3}\otimes u'_{4})=u_{5}\otimes u'_{4}, Sq^2(u_{4}\otimes u'_{3})=u_{4}\otimes u'_{5}$.

   Suppose that $C_{r}^5\wedge C_{r'}^{5}=X\vee Y$ and $H_{6}X\neq 0$, $C_{r}^5\wedge C_{r'}^{5}$ satisfies the conditions in Lemma 2.6,
  $$H_{t}X\cong H_{t}(C_{r}^5\wedge C_{r'}^{5}),~~t=6,9,10;~~~dim~H_{7}X+dim~H_{8}X\geq 2,$$
 which implies that $ \sum\limits_{t=1}^{\infty} dim~H_{t}Y\leq 1$.

If follows from  the isomorphism $Sq^2: \ch{7}{C_{r}^5\wedge C_{r'}^{5}}\rightarrow \ch{9}{C_{r}^5\wedge C_{r'}^{5}}$ that $Sq^2: \ch{7}{Y}\rightarrow \ch{9}{Y}$ is isomorphic. Hence $Y$ is not a Moore space with nontrivial homology group at 7 or 8 dimension. So $Y\simeq \ast$ and $C_{r}^5\wedge C_{r'}^{5}$ is indecomposable.
 \end{proof}

 In the rest of this subsection we will give cell structure of spaces $\m{r}{n}\wedge C_{\eta}^{n+2} (n\geq 3)$  and $C_{\eta}^5\wedge C_{r}^{5,s}$ which will be used later.

 \begin{lemma}\label{lemma4.2}
     $\m{r}{n}\wedge C_{\eta}^{n+2}\simeq S^{2n}\cup_{h^r}\CH C^{(2n+2),r}\simeq C_{r}^{2n+2}\cup_{h_r}\CH S^{2n+2}$ $(n\geq 3)$, where

   $h^r$ is determined by $h^ri_S=(2^r,\eta)$, i.e.,
$$\xymatrix{  C^{(2n+2),r}  \ar[r]^{h^r} &   S^{2n}   \ar[r]^{i_{\overline{C}}~~~} &   \m{r}{n}\wedge C_{\eta}^{n+2} \ar[r]^{~~~ q_{\overline{C}}}&   C^{(2n+3),r}  \\
  S^{2n}\vee S^{2n+1} \ar[u]_{i_S}\ar[ur]_{~~h^ri_S=(2^r, \eta)} };$$

     $h_r$ is determined by $q_{S}h_r=\mtwo{\eta}{2^r}$, i.e.,
$$\xymatrix{ S^{2n+2} \ar[dr]_{q_{S}h_r=\mtwo{\eta}{2^r}~~~~~~} \ar[r]^{h_r}
                & C_{r}^{2n+2} \ar[d]^{q_S}  \ar[r]^{i_{\underline{C}}} &   \m{r}{n}\wedge C_{\eta}^{n+2}\ar[r]^{~~~q_{\underline{C}}}&   S^{2n+3} \\
                & S^{2n+1}\vee S^{2n+2} }$$
  The top rows are cofibre sequences in each commutative diagrams.

Moreover,
       \begin{align}
    &\pi_{2n+1}(\m{r}{n}\wedge C_{\eta}^{n+2})=0;\label{Pi(2n+1)MCeta}\\
     & \z{r} \cong \frac{\pi_{2n+2}C_{r}^{2n+2}} {\lrg{(j_1\eta)_S^{-}, 2^{r-1}(2j_2)_S^{-}}}  \xrightarrow[(i_{\underline{C}})_{\ast}]{\cong}\pi_{2n+2}(\m{r}{n}\wedge C_{\eta}^{n+2});   \label{Pi(2n+2)MCeta}\\
   & \frac{\pi_{2n+3}S^{2n}} {\lrg{2^r\varrho_{2n}, \eta^{(3)}}}  \xrightarrow[(i_{\overline{C}})_{\ast}]{\cong}\pi_{2n+3}(\m{r}{n}\wedge C_{\eta}^{n+2});   \label{Pi(2n+3)MCeta}
\end{align}
Dually
   \begin{align}
    &[\m{r}{n}\wedge C_{\eta}^{n+2}, S^{2n+2}]=0;\label{CoPi(2n+2)MCeta}\\
     & \z{r} \cong \frac{[C^{(2n+3),r}, S^{2n+1}]} {\lrg{(\eta p_2)_S^{-}, 2^{r-1}(2p_1)_S^{-}}}  \xrightarrow[(q_{\underline{C}})^{\ast}]{\cong}[\m{r}{n}\wedge C_{\eta}^{n+2}, S^{2n+1}];   \label{CoPi(2n+1)MCeta}\\
   & \frac{[S^{2n+3}, S^{2n}]} {\lrg{2^r\varrho_{2n}, \eta^{(3)}}}  \xrightarrow[(q_{\overline{C}})^{\ast}]{\cong}[\m{r}{n}\wedge C_{\eta}^{n+2}, S^{2n}];   \label{CoPi(2n)MCeta}
\end{align}
 where $\eta^{(3)}=\eta\eta\eta$.
\end{lemma}
   \begin{proof}
   From Lemma \ref{lemma2.4}, $$\m{r}{n}\wedge C_{\eta}^{n+2}=S^n\wedge S^n\cup \CH(S^n\wedge S^n\vee S^n\wedge S^{n+1})\cup \CH \CH(S^n\wedge S^{n+1});$$
   $$(\m{r}{n}\wedge C_{\eta}^{n+2})/S^{2n}\simeq (S^{2n+1}\vee S^{2n+2})\cup_{\mtwo{\eta}{-2^r}}\CH S^{2n+2}\simeq C^{(2n+3),r};$$
   $$(\m{r}{n}\wedge C_{\eta}^{n+2})^{(2n+2)}\simeq S^{2n}\cup_{(2^r,\eta)}\CH(S^{2n}\vee S^{2n+1})\simeq C_{r}^{2n+2}.$$
   So, there are cofibre sequences
        $$\xymatrix{
  C^{(2n+2),r}  \ar[r]^{h^r} &   S^{2n} \ar[r]^{i_{\overline{C}}~~~} &   \m{r}{n}\wedge C_{\eta}^{n+2} \\
  S^{2n}\vee S^{2n+1} \ar[u]_{i_S}\ar[ur]_{~~h^ri_S=(a, x\eta)} };$$
       $$\xymatrix{
  S^{2n+2} \ar[dr]_{q_{S}h_r=\mtwo{y\eta}{b}~~~~~~} \ar[r]^{h_r}
                & C_{r}^{2n+2} \ar[d]^{q_S}  \ar[r]^{i_{\underline{C}}} &   \m{r}{n}\wedge C_{\eta}^{n+2} \\
                & S^{2n+1}\vee S^{2n+2} },$$
   where $a,b\in \ZH$ and $x,y\in\{0,1\}$.
  Since  the following two homomorphisms
$$[C^{(2n+2),r}, S^{2n}]\xrightarrow{ i_{S}^{*}}[S^{2n}\vee S^{2n+1}, S^{2n}]$$
$$[S^{2n+2}, C_{r}^{2n+2}]\xrightarrow{ ( q_{S})_{*}}[S^{2n+2},  S^{2n+1}\vee S^{2n+2}]$$
  are injective, $h^r$ and $h_r$ are determined by $h^ri_{S}$ and $q_{S}h_{r}$ respectively.

   By $H_{2n}(\m{r}{n}\wedge C_{\eta}^{n+2})=\z{r}$ and $\pi_{2n+1}(\m{r}{n}\wedge C_{\eta}^{n+2})\cong \pi_{2n+1}(C_{r}^{2n+2})=0$,
   $$a=2^r, x=1;$$
   By $H_{2n+2}(\m{r}{n}\wedge C_{\eta}^{n+2})=\z{r}$ and $[\m{r}{n}\wedge C_{\eta}^{n+2}, S^{2n+2}]\cong [C^{(2n+3),r}, S^{2n+2}]=0$
    $$b=2^r, y=1.$$
   Thus we prove the first part of this Lemma.  Now $\pi_{\ast}(\m{r}{n}\wedge C_{\eta}^{n+2}) (\ast=2n+1, 2n+2, 2n+3)$ and
   $[\m{r}{n}\wedge C_{\eta}^{n+2}, S^{m}] (m=2n, 2n+1, 2n+2)$ are easily obtained.
\end{proof}

 \begin{lemma}\label{lemma4.3}
  $C_{\eta}^5\wedge C_{r}^{5,s}$ is homotopy equivalent to the mapping cone of map\newline $S^9\xrightarrow{\mtwo{i_{\overline{C}}\varrho_6}{h_s}}\m{r}{3}\wedge C_{\eta}^5\vee C_{s}^{9}$ ($i_{\overline{C}}$ and $h_s$ are defined in Lemma $\ref{lemma4.2}$) and  $\pi_9(C_{\eta}^5\wedge C_{r}^{5,s})\cong\left\{
                     \begin{array}{ll}
                        \z{s+1}\oplus \zz, & \hbox{$r>1$} \\
                        \z{s}\oplus\zz, & \hbox{$r=1$}
                     \end{array}
                   \right.$ .
\end{lemma}
 \begin{proof}
Apply Lemma \ref{lemma2.4} to cofibre sequences $$S^4\xrightarrow{\eta}S_{a}^3\rightarrow C_{\eta}^5;~~ S^3\vee \m{s}{3}\xrightarrow{f=(2^r, \eta q)} S^4_{b}\rightarrow C_{r}^{5,s}.$$
    $$(C_{\eta}^5\wedge C_{r}^{5,s})/S^6\simeq (\Sigma(S^4\wedge S_{b}^3)\vee \Sigma(S_a^3\wedge(S^3\vee \m{s}{3})))\cup_{\mathcal{A}}\CH \Sigma S^4\wedge(S^3\vee \m{s}{3})$$
where $\mathcal{A}=\mtwo{\Sigma 1\wedge f}{-\Sigma \eta \wedge 1}$. i.e.,
   $$(C_{\eta}^5\wedge C_{r}^{5,s})/S^6\simeq (S^8\vee S^7\vee\m{s}{7})\cup_{\mathcal{A}}\CH (S^8\vee\m{s}{8})$$
where $\mathcal{A}=\small{\begin{tabular}{c|cc|}
     \multicolumn{1}{c}{}  & $S^8$ & \multicolumn{1}{c}{$\m{s}{8}$}\\
       \cline{2-3}
  $S^8$ & $2^r$  & $\eta q$\\
    $S^{7}$ & $\eta$  & 0\\
  $\m{s}{7}$ & 0 &$\eta\wedge 1$ \\
     \cline{2-3}
   \end{tabular}}\xrightarrow[q\rrr{3}+\rrr{1}]{\cong}\small{\begin{tabular}{c|cc|}
     \multicolumn{1}{c}{}  & $S^8$ & \multicolumn{1}{c}{$\m{s}{8}$}\\
       \cline{2-3}
  $S^8$ & $2^r$  & 0\\
    $S^{7}$ & $\eta$  & 0\\
  $\m{s}{7}$ & 0 &$\eta\wedge 1$ \\
     \cline{2-3}
   \end{tabular}}$. Hence
    \begin{align}
        (C_{\eta}^5\wedge C_{r}^{5,s})/S^6 \simeq \m{s}{4}\wedge C_{\eta}^5\vee C^{9,r}  \label{PinchCetaCrs}
   \end{align}
   Apply Lemma \ref{lemma2.4} to cofibre sequences $$ S_{1}^4\xrightarrow{\eta} S^3\rightarrow C_{\eta}^5 ;~~ S_2^4\xrightarrow{g=\mtwo{i\eta}{2^s}}\m{r}{3}\vee S_a^4\rightarrow C_{r}^{5,s}.$$
$$\begin{array}{rl}
    (C_{\eta}^5\wedge C_{r}^{5,s} )^{(9)} & \simeq S^3\wedge (\m{r}{3}\vee S_a^4) \cup_{(\eta\wedge 1, ~1\wedge g)}\CH(S_1^4\wedge (\m{r}{3}\vee S_a^4) \vee S^3\wedge S_2^4)  \\
     &\simeq (\m{r}{6}\vee S^7)\cup_{\mathcal{B}}\CH(\m{r}{7}\vee S^8\vee S^7)
  \end{array}$$
where $\mathcal{B}=\small{\begin{tabular}{c|ccc|}
     \multicolumn{1}{c}{}  & $\m{r}{7}$ & $S^{8}$&\multicolumn{1}{c}{$S^7$}\\
       \cline{2-4}
  $\m{r}{6}$ & $\eta\wedge 1$ & $0$& $i\eta$\\
   $S^{7}$ & $0$  &$\eta$& $2^s$\\
     \cline{2-4}
   \end{tabular}} \xrightarrow[\ccc{1}i+\ccc{3}]{\cong}\small{\begin{tabular}{c|ccc|}
     \multicolumn{1}{c}{}  & $\m{r}{7}$ & $S^{8}$&\multicolumn{1}{c}{$S^7$}\\
       \cline{2-4}
  $\m{r}{6}$ & $\eta\wedge 1$ & $0$& $0$\\
   $S^{7}$ & $0$  &$\eta$& $2^s$\\
     \cline{2-4}
   \end{tabular}}$ . Thus
     \begin{align}
        (C_{\eta}^5\wedge C_{r}^{5,s})^{(9)}\simeq \m{r}{3}\wedge C_{\eta}^5\vee C_{s}^9  \label{codimCetaCrs}
   \end{align}
   There is a cofibre sequence
\begin{align}
S^9\xrightarrow{\mtwo{\alpha'}{\beta'}}\m{r}{3}\wedge C_{\eta}^5\vee C_{s}^{9}\rightarrow C_{\eta}^5\wedge C_{r}^{5,s}\rightarrow S^{10}\xrightarrow{\mtwo{\Sigma\alpha'}{\Sigma\beta'}}\m{r}{4}\wedge C_{\eta}^5\vee C_{s}^{10} \label{CofCetaCrs}
\end{align}
   where $\alpha'=i_{\overline{C}}(t'\varrho_6)$, $t'=1$ for $r=1$ and $t'\in\{1,2\}$ for $r>1$ since $\pi_{9}(\m{r}{3}\wedge C_{\eta}^5)=\left\{
                                       \begin{array}{ll}
                                         \zz\lrg{i_{\overline{C}}\varrho_6}, & \hbox{$r=1$} \\
                                         \ZH/4\lrg{i_{\overline{C}}\varrho_6}, & \hbox{$r\geq 2$}
                                       \end{array}
                                     \right. $ and    $C_{\eta}^5\wedge C_{r}^{5,s}$ is indecomposable. $\beta'$ is determined by
   $q_{S}\beta'=\mtwo{y'\eta}{b'}$ for some $y'\in\{0,1\}$ and $b'\in\ZH$ since $(q_{S})_{\ast}:\pi_{9}C_{s}^9\rightarrow \pi_{9}(S^8\vee S^9)$ is injective.
$$\xymatrix{    S^{6}   \ar[r]^{i_{\overline{C}}~~~} &   \m{r}{3}\wedge C_{\eta}^{5} \\
  S^{9} \ar[u]_{t'\varrho_6}\ar[ur]_{\alpha' }}~~~~~~~~~~~~~~~~~~\xymatrix{
  S^9 \ar[dr]_{\mtwo{y'\eta}{b'}~~~~} \ar[r]^{\beta'} & C_{s}^9 \ar[d]^{q_S}  \\
                & S^8\vee S^9             }$$
By $H_{9}(C_{\eta}^5\wedge C_{r}^{5,s})=\z{s}$, $b'=2^s$.
From $[C_{\eta}^5\wedge C_{r}^{5,s}, S^9]\cong [(C_{\eta}^5\wedge C_{r}^{5,s})/S^6, S^9]=[\m{s}{4}\wedge C_{\eta}^5\vee C^{9,r}, S^9]\cong \z{r}$ and exact sequence  $0\rightarrow\frac{[S^{10}, S^9]}{\lrg{y'\eta}}\rightarrow [C_{\eta}^5\wedge C_{r}^{5,s}, S^9]\rightarrow \z{r}\rightarrow 0$ we have $y'=1$. Thus $\beta'=h_{s}.$  So for $r=1$, 
$$ \pi_{9}(C_{\eta}^5\wedge C_{r}^{5,s})\cong\frac{\ZH/2\lrg{i_{\overline{C}}\varrho_6}\oplus \zz\lrg{(i_{1}\eta)^-}\oplus \ZH \lrg{(2i_2)^-}}{\lrg{(t'i_{\overline{C}}\varrho_6,(i_{1}\eta)^-,2^{s-1}(2i_2)^-)}} \cong \frac{\ZH/4\oplus\zz\oplus\ZH}{\lrg{(1,1,2^{s-1})}}\cong \ZH/2\oplus\z{s}.$$

Next we will determine  $t'$  for $r>1$.

By computing the exact sequence $\pi_{9}$ of  cofibre sequence (\ref{CofCetaCrs}), we get
 \begin{align}
& \pi_{9}(C_{\eta}^5\wedge C_{r}^{5,s})\cong\frac{\ZH/4\lrg{i_{\overline{C}}\varrho_6}\oplus \zz\lrg{(i_{1}\eta)^-}\oplus \ZH \lrg{(2i_2)^-}}{\lrg{(t'i_{\overline{C}}\varrho_6,(i_{1}\eta)^-,2^{s-1}(2i_2)^-)}}  \nonumber\\
&\cong \frac{\ZH/4\oplus\zz\oplus\ZH}{\lrg{(t',1,2^{s-1})}}\cong \left\{
                                                             \begin{array}{ll}
                                                               \ZH/4\oplus \z{s}, & \hbox{$t'=2$} \\
                                                                \ZH/2\oplus \z{s+1}, & \hbox{$t'=1$}
                                                             \end{array}
                                                           \right.. \label{pi9CetaCrs1}
\end{align}
On the other hand, $\pi_9(C_{\eta}^5\wedge C_{r}^{5,s})\cong [C_{\eta}^{14}, C_{r}^{13,s}]\cong [C_{\eta}^7, C_{r}^{6,s}]$ ($[C_{\eta}^{k+1}, C_{r}^{k,s}]$ is stable for $k\geq 6$). There is an exact sequence
    $$[C_{\eta}^{7}, S^5]\xrightarrow{\mtwo{i\eta}{2^s}_{\ast}}  [C_{\eta}^{7}, \m{r}{4}\vee S^5]\rightarrow[C_{\eta}^7, C_{r}^{6,s}]\rightarrow [C_{\eta}^{7}, S^6]=0.$$
   From Lemma \ref{lemma3.1} and the following commutative diagram
 $$\xymatrix{
  [S^6, S^4] \ar[r]^{\eta^{\ast}} \ar[d]^{i_{\ast}} &  [S^7, S^4] \ar[r]^{q_{\eta}^{\ast}} \ar@{->>}[d]^{i_{\ast}} &   [C_{\eta}^7, S^4] \ar[r]^{0} \ar[d]^{i_{\ast}} &   [S^5, S^4] \ar[r]^{\eta^{\ast}}_{\cong} \ar[d]^{i_{\ast}} & [S^6, S^4]\ar[d]^{i_{\ast}}  \\
             [S^6, \m{r}{4}] \ar[r]^{\eta^{\ast}}  &  [S^7, \m{r}{4}] \ar[r]^{q_{\eta}^{\ast}} &   [C_{\eta}^7, \m{r}{4}] \ar[r]^{0} &   [S^5, \m{r}{4}] \ar[r]^{\eta^{\ast}}_{\cong}  & [S^6, \m{r}{4}] }$$
  we get the surjection  $$\ZH/12\cong [C_{\eta}^7, S^4]\xrightarrow{i_{\ast}}[C_{\eta}^7, \m{r}{4}]\cong \ZH/4.$$
  By
   $[C_{\eta}^7, C_{\eta}^6]\cong \frac{[C_{\eta}^7, S^4]}{\eta_{\ast}[C_{\eta}^7, S^5]}\cong \frac{[S^7, C_{\eta}^6]}{\eta^{\ast}[S^6, C_{\eta}^6]}\cong \ZH/6 $ ~(\cite{RefUnsold}, Proposition~2.6 (iii)), we get
  $$\ZH=[C_{\eta}^7, S^5]\xrightarrow{\eta_{\ast}=6}[C_{\eta}^7, S^4]=\ZH/12.$$

So
\begin{align}
(i\eta)_{\ast}=2:  \ZH=[C_{\eta}^7, S^5]\xrightarrow{\eta_{\ast}}[C_{\eta}^7, S^4]\xrightarrow{i_{\ast}}[C_{\eta}^7, \m{r}{4}]\cong \ZH/4. \label{map(ieta)}
\end{align}
 hence
 \begin{align}
 \pi_9(C_{\eta}^5\wedge C_{r}^{5,s})\cong[C_{\eta}^7, C_{r}^{6,s}]\cong \frac{\ZH/4\oplus \ZH}{\lrg{(2,2^s)}}\cong \z{s+1}\oplus \zz. \label{pi9C(CetaCrs)}
\end{align}
 Together with (\ref{pi9CetaCrs1}), $t'=1$.
\end{proof}

\subsection{$\m{u}{3}\wedge C$, $C\in \{C_{r}^{5}, C^{5,s}, C_{r}^{5,s}~|~r,s\in\ZH^{+}\}$}
\label{subsec4.2}
\begin{itemize}

  \item [(1)] $\m{u}{3}\wedge C_{r}^5$ and $\m{u}{3}\wedge C^{5,s}$

   There is the following commutative diagram

   $$\xymatrix@C=0.5cm{
 \m{u}{3}\wedge (S^3\vee S^4)\ar@{=}[d] \ar[rr]^{~~1\wedge (2^r, \eta)}& & \m{u}{3}\wedge S^3 \ar[rr] & &\m{u}{3}\wedge C_{r}^5 \\
 \m{u}{3}\wedge S^3\vee \m{u}{3}\wedge S^4 \ar[rru]_{~~(1\wedge 2^r, 1\wedge\eta)}&&  }$$
  where the top row is a cofibre sequence. Hence
     $$\m{u}{3}\wedge C_{r}^5\simeq\m{u}{3}\wedge S^3\cup_{(1\wedge 2^r, 1\wedge\eta)}\CH(\m{u}{3}\wedge S^3\vee\m{u}{3}\wedge S^4 )$$

\begin{itemize}
  \item [$\bullet$] If $r\geq u$ and $r>1$, then $1\wedge 2^r\simeq 0$, hence $\m{u}{3}\wedge C_{r}^5\simeq \m{u}{7}\vee \m{u}{3}\wedge C_{\eta}^5.$

\item [$\bullet$]  If $r=u=1$, $1\wedge\eta\in [\m{u}{7},\m{u}{6}]^{\eta}_{\eta}$ implies  $1\wedge \eta=\lambda_{11}+\varepsilon_{u}i\eta\eta q$ for some $\varepsilon_{u}\in \{0.1\}$ ( for $u=1$, $\varepsilon_{1}=0$, i.e., $\lambda_{11}=1\wedge \eta$ ),
        then $$\moo{3}\wedge C_{1}^5\simeq\moo{6}\cup_{\mathcal{M}=(i_6\eta q_6, ~\lambda_{11})}\CH(\moo{6}\vee\moo{7} ).$$

        From $\footnotesize{\begin{tabular}{c|cc|}
     \multicolumn{1}{c}{}  & $\moo{6}$ &  \multicolumn{1}{c}{$\moo{7}$} \\
       \cline{2-3}
     $\moo{6}$ & $i_6\eta q_6$ & $\lambda_{11}$ \\
     \cline{2-3}
   \end{tabular} \begin{tabular}{c|cc|}
     \multicolumn{1}{c}{}  & $\moo{6}$ &  \multicolumn{1}{c}{$\moo{7}$} \\
       \cline{2-3}
     $\moo{6}$ & $id$ & $0$ \\
     $\moo{7}$ & $i_7q_6$ & $id$ \\
 \cline{2-3}
       \multicolumn{3}{c}{} \\
   \end{tabular}= \begin{tabular}{c|cc|}
     \multicolumn{1}{c}{}  & $\moo{6}$ &  \multicolumn{1}{c}{$\moo{7}$} \\
       \cline{2-3}
     $\moo{6}$ & $0$ & $\lambda_{11}$ \\
     \cline{2-3}
   \end{tabular}}$~, we get
\newline
   $\mathcal{M}\cong (0, ~\lambda_{11})$, thus
 $\moo{3}\wedge C_{1}^5\simeq \moo{7}\vee \moo{3}\wedge C_{\eta}^5.$

 \item [$\bullet$]  If $r<u$,   apply Lemma \ref{lemma2.4} to cofibre sequences
      $$S_1^3\xlongrightarrow{2^u}S_a^3\rightarrow\m{u}{3};~~S_2^3\vee S^4\xlongrightarrow{(2^r, ~\eta)}S_b^3\rightarrow  C_{r}^{5},$$
   we can easily get
 \begin{align}
(\m{u}{3}\wedge C_{r}^5)/S^6\simeq  S^7\vee C^{9,u}_{r} \label{modS6 1}
\end{align}
   Suppose that  $\m{u}{3}\wedge C_{r}^5\simeq X\vee Y$ are decomposable
   and $H_{6}(X)\neq 0$.

  Since $Sq^2$ on $\ch{6}{\m{u}{3}\wedge C_{r}^5}$ and $\ch{7}{\m{u}{3}\wedge C_{r}^5}$ are nontrivial and
   $$H_{\ast}(\m{u}{3}\wedge C_{r}^5)=\left\{
                                           \begin{array}{ll}
                                             \z{r}, & \hbox{$\ast=6$} \\
                                             \z{r}, & \hbox{$\ast=7$} \\
                                              \z{u}, & \hbox{$\ast=8$} \\
                                             0, & \hbox{otherwise}
                                           \end{array}
                                         \right.~, $$

     $\m{u}{3}\wedge C_{r}^5$ has no direct summands $\m{r}{6}$ and $\m{u}{8}$. Thus
     $\m{u}{3}\wedge C_{r}^5\simeq X\vee \m{r}{7}$. By Lemma \ref{lemma2.7} and (\ref{modS6 1}),
      $(X/S^6)\vee \m{r}{7}\simeq S^7\vee C^{9,u}_{r}$ which contradicts to the uniqueness of the decomposability of $\mathbf{A}_{n}^{3}$-complexes \cite{RefBH}. So $\m{u}{3}\wedge C_{r}^5$ is indecomposable for $r<u$.
\end{itemize}

In summary,  $M_{2^{u}}^3\wedge C_{r}^5$ is homotopy equivalent to $ M_{2^{u}}^3\wedge C_{\eta}^{5} \vee M_{2^{u}}^7$ for $r\geq u$ and indecomposable for $u>r$.

   By the properties of the duality functor $D$ we have

      $M_{2^{u}}^3\wedge C^{5,s}\simeq D(M_{2^{u}}^4\wedge C_{s}^5)$ is homotopy equivalent to $ M_{2^{u}}^3\wedge C_{\eta}^{5} \vee M_{2^{u}}^7$ for $s\geq u$ and indecomposable for $u>s$.

 \item [(2)] $\m{u}{3}\wedge C_{r}^{5,s}$

   If $u>r$ and $u>s$  there is a cofibre sequence
   $$S^3\wedge C_{r}^{5,s}\xrightarrow{2^u\wedge 1}S^3\wedge C_{r}^{5,s}\rightarrow \m{u}{3}\wedge C_{r}^{5,s} $$
 $[C^{k,s}_{r}, C^{k,s}_{r}]\cong\z{max(r,s)+1}\oplus \z{m_{sr}}\oplus \z{m_{sr}+1}$ for $k\geq 5$ which implies $2^u\wedge 1=0\in [S^3\wedge C_{r}^{5,s}, S^3\wedge C_{r}^{5,s}]$. Thus
   $ \m{u}{3}\wedge C_{r}^{5,s}\simeq C_{r}^{8,s}\vee C_{r}^{9,s}.$

   If $u\leq r$ and $r>1$, by \textbf{Cof4} of $C_{r}^{5,s}$,
   $$\xymatrix@C=0.5cm{
 \m{u}{3}\wedge (S^3\vee \m{s}{3})\ar@{=}[d] \ar[rr]^{~~1\wedge (2^r, \eta q)}& & \m{u}{3}\wedge S^3 \ar[rr] & &\m{u}{3}\wedge C_{r}^{5,s} \\
 \m{u}{3}\wedge S^3\vee \m{u}{3}\wedge \m{s}{3} \ar[rru]_{~~(1\wedge 2^r, 1\wedge\eta q)}&&  }$$
  where the top row is a cofibre sequence. Since $2^r=0\in [\m{u}{6}, \m{u}{6}]$,
     $$\m{u}{3}\wedge C_{r}^{5,s}\simeq\m{u}{7}\vee \CH_{1\wedge \eta q}\simeq \m{u}{3}\wedge C^{5,s}\vee \m{u}{7}$$
   $$ \m{u}{3}\wedge C_{r}^{5,s}\simeq \left\{
                                        \begin{array}{ll}
                                         \m{u}{3}\wedge C^{5,s}\vee \m{u}{7}, &\hbox{$s<u\leq r$} \\
                                          \m{u}{3}\wedge C_{\eta}^{5}\vee \m{u}{7}\vee \m{u}{7}, &  \hbox{$u\leq r,s$ and $r>1$}
                                        \end{array}
                                      \right. .$$

 If $u\leq s$ and $s>1$, $\m{u}{3}\wedge C_{r}^{5,s}\simeq D(\m{u}{4}\wedge C_{s}^{5r})$, we get

   $$ \m{u}{3}\wedge C_{r}^{5,s}\simeq \left\{
                                        \begin{array}{ll}
                                         \m{u}{3}\wedge C_{r}^{5}\vee \m{u}{7}, & \hbox{$r<u\leq s$} \\
                                          \m{u}{3}\wedge C_{\eta}^{5}\vee \m{u}{7}\vee \m{u}{7}, & \hbox{$u\leq r,s$ and $s>1$}
                                        \end{array}
                                      \right. .$$

  If $u=r=s=1$, then from Corollary 3.7 of \cite{RefWJ},
    $$\moo{3}\wedge C_{1}^{5,1}\simeq \Sigma \moo{3}\wedge \moo{1}\wedge \moo{1}\simeq  \moo{3}\wedge C_{\eta}^{5}\vee \moo{7}\vee \moo{7}.$$

    In summary,
   $$\m{u}{3}\wedge C_{r}^{5,s}\simeq \left\{
                                        \begin{array}{ll}
                                          C_{r}^{8,s}\vee C_{r}^{9,s}, &\hbox{$u>s$ and $u>r$} \\
                                         \m{u}{3}\wedge C_{r}^{5}\vee \m{u}{7}, & \hbox{$r<u\leq s$} \\
                                          \m{u}{3}\wedge C^{5,s}\vee \m{u}{7}, &\hbox{$s<u\leq r$} \\
                                          \m{u}{3}\wedge C_{\eta}^{5}\vee \m{u}{7}\vee \m{u}{7}, & \hbox{$u\leq r,s$}
                                        \end{array}
                                      \right. .$$
    \end{itemize}

 \subsection{$C_{u}^5\wedge C_{r}^{5,s}$ and  $C^{5,u}\wedge C_{r}^{5,s}$ for $u,r,s\in\ZH^{+}$}
\label{subsec4.3}

 Let $u_{k}$ be generators of $H^{k}(C_{u}^{5}; \zz)$  for $k=3,4,5$ and  $v_{3}, v_{4}, \overline{v_{4}}, v_{5}$ be generators of $H^{\ast}(C^{5,s}_r; \zz)$ which satisfy conditions (\ref{Sq1}) of Lemma \ref{lemma3.3}.
 $$H_{\ast}(C^{5}_{u}\wedge C^{5,s}_{r})=\left\{
                  \begin{array}{ll}
           \mathbb{Z}/2^{m_{ur}}& \hbox{$\ast=6$} \\
          \mathbb{Z}/2^{m_{ur}}\oplus  \mathbb{Z}/2^{m_{us}} & \hbox{$\ast=7$} \\
         \mathbb{Z}/2^{m_{us}}\oplus \z r & \hbox{$\ast=8$} \\
         \z s & \hbox{$\ast=9$} \\
           0 & \hbox{ otherwise}
             \end{array}
                    \right.$$

  $$H^{\ast}(C^{5}_{u}\wedge C^{5,s}_{r}; \zz)=\left\{
                                         \begin{array}{ll}
                                           \mathbb{Z}/2\{u_3\otimes v_3\} & \hbox{$\ast=6$} \\
                                            \mathbb{Z}/2\{u_3\otimes v_4, u_3\otimes \overline{v_4}, u_4\otimes v_3 \} & \hbox{$\ast=7$} \\
                                            \mathbb{Z}/2\{u_3\otimes v_5, u_4\otimes v_4, u_4\otimes \overline{v_4}, u_5\otimes v_3\} & \hbox{$\ast=8$} \\
                                             \mathbb{Z}/2\{u_4\otimes v_5, u_5\otimes v_4, u_5\otimes \overline{v_4}\} & \hbox{$\ast=9$} \\
                                            \mathbb{Z}/2\{u_5\otimes v_5\} & \hbox{ $\ast=10$}\\
                                             0& \hbox{otherwise}
                                         \end{array}
                                       \right.$$

  The Steenrod operation action on $H^{\ast}(C^{5}_{u}\wedge C^{5,s}_{r}; \zz)$ is given as follows

\begin{itemize}
     \item [(i)] $Sq^4(u_{3}\otimes v_{3})=u_{5}\otimes v_{5}$;
     \item [(ii)] $Sq^2(u_{3}\otimes v_{5})=Sq^2(u_{5}\otimes v_{3})=u_{5}\otimes v_{5}$;
     \item [(iii)]$Sq^2(u_{3}\otimes v_{3})=  \left\{
                                                \begin{array}{ll}
                                                  u_{3}\otimes v_{5}+u_{4}\otimes v_{4}+u_{5}\otimes v_{3}, & \hbox{$u=r=1$} \\
                                                  u_{3}\otimes v_{5}+u_{5}\otimes v_{3}, & \hbox{$u>1$ or $r>1$}
                                                \end{array}
                                             \right.$;
     \item [(iv)]
     \begin{itemize}
        \item [] $Sq^2(u_3\otimes v_4)= u_5\otimes v_4$;
        \item []$Sq^2(u_{4}\otimes v_{3})=u_{4}\otimes v_{5}$;
        \item []  $Sq^2(u_3\otimes \overline{v_4})= \left\{
                                                              \begin{array}{ll}
                                                            u_5\otimes \overline{v_4}+u_4\otimes v_5, & \hbox{$u=s=1$} \\
                                                           u_5\otimes \overline{v_4}, & \hbox{$u>1$ or $s>1$}
                                                              \end{array}
                                                    \right.$.
         \end{itemize}
 \end{itemize}

\begin{corollary}\label{corollary4.4}
 If $C^{5}_{u}\wedge C^{5,s}_{r}\simeq X\vee Y$ is decomposable and $H_{6}(X)\neq 0$, then  $X$ is indecomposable and $Y\simeq C^{9,t}_l$ for some $t\in  \{m_{us},r\}, l\in \{m_{ur},m_{us}\}$.
  \end{corollary}
\begin{proof}
  From Lemma \ref{lemma2.6},  $H_{t}X\cong H_{t}(C^{5}_{u}\wedge C^{5,s}_{r})$ for $t=6,9$ and $dim~H_{7}X+dim~H_{8}X\geq 2$. Neither $\m{x}{7}$ nor $\m{x}{8}$ can be a summand of $Y$ for any $x\in \ZH^+$ since
    $$Sq^2:  H^7(C^{5}_{u}\wedge C^{5,s}_{r}; \zz)\rightarrow   H^9(C^{5}_{u}\wedge C^{5,s}_{r}; \zz)$$ is isomorphic.
  Thus $dim~H_{7}X=dim~H_{8}X=1$ and $dim~H_{\ast}Y=\left\{
                                                   \begin{array}{ll}
                                                     1, & \hbox{$\ast=7,8$} \\
                                                     0, & \hbox{otherwise} \\
                                                   \end{array}
                                                 \right.$ which implies that
 $X$ is indecomposable and $Y\simeq C^{9,t}_l$ for some $t\in  \{m_{us},r\}, l\in \{m_{ur},m_{us}\}$.
\end{proof}

Firstly we study $(\CC)/S^{6}$.

Applying Lemma \ref{lemma2.4} to the following  cofibre sequences
 $$S_{1}^{3}\vee S^4\xlongrightarrow {(2^u, \eta)}S_{a}^3\rightarrow C_{u}^5;~~~~~S_{2}^3\vee\m{s}{3}\xlongrightarrow {(2^r, \eta q)}S_{b}^3\rightarrow C_{r}^{5,s},$$
$$(\CC)/S^{6}\simeq(\Sigma(S_{1}^3\vee S^{4})\wedge S_{b}^3\vee\Sigma S_{a}^3\wedge(S_{2}^3\vee \m{s}{3}))\cup_{\mathcal{A}}\CH\Sigma (S_{1}^3\vee S^{4})\wedge(S_{2}^3\vee \m{s}{3}),$$
 where $\mathcal{A}= \mtwo{\Sigma 1 \wedge (2^r,\eta q)}{-\Sigma (2^u,\eta)\wedge 1)}$, i.e.,
$$(\CC)/S^{6}\simeq (S^7\vee S^8\vee S^{7}\vee\m{s}{7})\cup_{\mathcal{A}}\CH(S^7\vee S^8\vee\m{s}{7}\vee\m{s}{8} )$$
$$\mathcal{A}=\footnotesize{\begin{tabular}{c|cccc|}
     \multicolumn{1}{c}{}  & $S^7$ & $S^8$&$\m{s}{7}$&\multicolumn{1}{c}{$\m{s}{8}$}\\
       \cline{2-5}
    $S^7$ & $2^r$  & 0 &$\eta q$&0\\
   $S^8$ & $0$  &$2^r$& 0 &$\eta q$\\
   $S^7$&$-2^u$ & $\eta$  & 0  &0\\
   $\m{s}{7}$&0 & 0  & $-2^u$  &$\eta\wedge 1$\\
     \cline{2-5}
\end{tabular}}.$$

Note that $q(\eta\wedge 1)=\eta q$ for  $u\geq s$; $2^{u}=\left\{
         \begin{array}{ll}
           0, & \hbox{$u>1$} \\
           i\eta q, & \hbox{$u=s=1$}
         \end{array}
       \right.$   in $[\m{s}{7}, \m{s}{7}]$.

\begin{itemize}
   \item [(i)] For $s\leq u<r$, by transformations $i\rrr{1}+\rrr{4}$ if $u=1$ (otherwise, omitting this one); $ q\rrr{4}+\rrr{2}$;  $2^{r-u}\rrr{3}+\rrr{1}$ and $-\rrr{3}$,
 $$\mathcal{A}\cong\footnotesize{\begin{tabular}{c|cccc|}
     \multicolumn{1}{c}{}  & $S^7$ & $S^8$&$\m{s}{7}$&\multicolumn{1}{c}{$\m{s}{8}$}\\
       \cline{2-5}
    $S^7$ & 0  & 0 &$\eta q$&0\\
   $S^8$ & $0$  &$2^r$& 0 &0\\
   $S^7$&$2^u$ & $\eta$  & 0  &0\\
   $ \m{s}{7}$&0 & 0  & 0 &$\eta\wedge 1$\\
     \cline{2-5}
\end{tabular}}$$
 Thus
 \begin{align}
    (\CC)/S^6\simeq C_{u}^{9,r}\vee C^{9,s}\vee \m{s}{4}\wedge C_{\eta}^{5}. \label{Pinch1}
\end{align}

 \item [(ii)] For $s\leq u$ and $r\leq u$, by transformations  $i\rrr{1}+\rrr{4}$ if $u=1$; $q\rrr{4}+\rrr{2}$; $2^{u-r}\rrr{1}+\rrr{3}$;
$\ccc{2}q+\ccc{3}$ if $r=u$,
 $$\mathcal{A}\cong\footnotesize{\begin{tabular}{c|cccc|}
     \multicolumn{1}{c}{}  & $S^7$ & $S^8$&$\m{s}{7}$&\multicolumn{1}{c}{$\m{s}{8}$}\\
       \cline{2-5}
    $S^7$ & $2^r$  & 0 &$\eta q$&0\\
   $S^8$ & $0$  &$2^r$& 0 &0\\
   $S^7$&0 & $\eta$  & 0  &0\\
   $ \m{s}{7}$&0 & 0  & 0 &$\eta\wedge 1$\\
     \cline{2-5}
\end{tabular}}$$
Thus
 \begin{align}
    (\CC)/S^6\simeq C_{r}^{9,s}\vee C^{9,r}\vee \m{s}{4}\wedge C_{\eta}^{5}.               \label{Pinch2}
\end{align}

 \item [(iii)] For $r\leq u<s$  by transformations $ q\rrr{4}+\rrr{2}$; $2^{u-r}\ccc{2}q+\ccc{3}$; $2^{u-r}\rrr{1}+\rrr{3}$ and $-\rrr{4}$,
$$\mathcal{A}\cong\footnotesize{\begin{tabular}{c|cccc|}
     \multicolumn{1}{c}{}  & $S^7$ & $S^8$&$\m{s}{7}$&\multicolumn{1}{c}{$\m{s}{8}$}\\
       \cline{2-5}
    $S^7$ & $2^r$  & 0 &$\eta q$&0\\
   $S^8$ & $0$  &$2^r$& 0 &0\\
   $S^7$&0 & $\eta$  & 0  &0\\
   $ \m{s}{7}$&0 & 0  & $2^u$  &$\eta\wedge 1$\\
     \cline{2-5}
\end{tabular}}~.$$
Thus
\begin{align}
    (\CC)/S^6\simeq C^{9,r}\vee Z                      \label{Pinch3}
\end{align}
where $Z=(S^7\vee \m{s}{7})\cup_{\left(
                                   \footnotesize{\begin{array}{ccc}
                                     2^r & \eta q & 0 \\
                                     0 & 2^u & \eta\wedge 1 \\
                                   \end{array}}
                                 \right)}\CH(S^7\vee\m{s}{7}\vee\m{s}{8})$.

\item [(iv)] For $u<s$ and $u<r$, by transformations  $2^{r-u}\rrr{3}+\rrr{1}$; $-\rrr{3}$ and $-\rrr{4}$,
  \begin{align}
\mathcal{A}\cong\footnotesize{\begin{tabular}{c|cccc|}
     \multicolumn{1}{c}{}  & $S^7$ & $S^8$&$\m{s}{7}$&\multicolumn{1}{c}{$\m{s}{8}$}\\
       \cline{2-5}
    $S^7$ & 0 & 0 &$\eta q$&0\\
   $S^8$ & $0$  &$2^r$& 0 &$\eta q$\\
   $S^7$&$2^u$ & $\eta$  & 0  &0\\
   $ \m{s}{7}$&0 & 0  & $2^u$  &$\eta\wedge 1$\\
     \cline{2-5}
\end{tabular}}~.\label{Pinch4}
\end{align}

 \end{itemize}

Secondly, we study the codimension 1 skeleton  $(\CC)^{(9)}$ of $\CC$.

Using cofibre sequences  $$S_{1}^4\xlongrightarrow{i\eta}\m{u}{3}\rightarrow C^{5}_u;  ~~~~S_{2}^4\xlongrightarrow{g=\matt{i\eta}{s}}\m{r}{3}\vee S_{b}^4\rightarrow C_{r}^{5,s}$$

$$(C^{5}_u\wedge C_{r}^{5,s})^{(9)}\simeq \m{u}{3}\wedge(\m{r}{3}\vee S_{b}^4)\cup_{\mathcal{B}=(i\eta\wedge 1, 1\wedge g)}\CH(S_{1}^4\wedge(\m{r}{3}\vee S_{b}^4)\vee\m{u}{3}\wedge S_{2}^4)$$

i.e., $(C^{5}_u\wedge C_{r}^{5,s})^{(9)}\simeq (\m{u}{3}\wedge S_{b}^4 \vee \m{u}{3}\wedge\m{r}{3})\cup_{\mathcal{B}}\CH(S_{1}^4\wedge S_{b}^4\vee S_{1}^4\wedge\m{r}{3} \vee\m{u}{3}\wedge S_{2}^4)$,
where $$\mathcal{B}=\footnotesize{\begin{tabular}{c|ccc|}
     \multicolumn{1}{c}{}  & $S_{1}^4\wedge S_{b}^4$ & $S_{1}^4\wedge \m{r}{3}$&\multicolumn{1}{c}{$\m{u}{3}\wedge S_{2}^4$}\\
       \cline{2-4}
    $\m{u}{3}\wedge S^{4}_{b}$ &  $i\eta\wedge 1$ & $i\eta\wedge 0$  &$1\wedge 2^s$\\
   $\m{u}{3}\wedge\m{r}{3}$ & $i\eta\wedge 0$   &$i\eta\wedge 1$ &$1\wedge i\eta$\\
       \cline{2-4}
   \end{tabular}}~.$$

Let  $$\mathcal{B}_1=\footnotesize{\begin{tabular}{c|cc|}
     \multicolumn{1}{c}{}  & $S_{1}^4\wedge \m{r}{3}$ & \multicolumn{1}{c}{$\m{u}{3}\wedge S_{2}^4$}\\
       \cline{2-3}
   $\m{u}{3}\wedge\m{r}{3}$ & $i\eta\wedge 1$  & $1\wedge i\eta$\\
     \cline{2-3}
   \end{tabular}}~.$$

  For $u\geq r$ and $u>1$, by Corollary \ref{corollary3.2}, there is a retraction $\tau'_3$ of $i\wedge 1$ yielding the following commutative diagram
     $$\xymatrix{
    &S^4\wedge\m{r}{3}\ar[d]^{i\eta\wedge 1}\ar[dl]_{\eta\wedge 1}\ar[dr]^{0}& \\
    S^3\wedge\m{r}{3}\ar[r]^{i\wedge 1}&\m{u}{3}\wedge\m{r}{3}\ar@/^/[l]^{\tau'_3}\ar[r]^{q\wedge 1}& S^4\wedge\m{r}{3}
      }~$$

 From the following commutative diagrams
$$\xymatrix{
\m{u}{3}\wedge S^4\ar[r]^{q\wedge 1}\ar[d]^{1\wedge \eta}& S^4\wedge S^4\ar[d]^{1\wedge \eta}\\
\m{u}{3}\wedge S^3\ar[r]^{q\wedge 1}\ar[d]^{1\wedge i}& S^4\wedge S^3\ar[d]^{1\wedge i}\\
\m{u}{3}\wedge\m{r}{3}\ar[r]^{q\wedge 1}& S^4\wedge\m{r}{3}
} ~~~~~~~~~~~~~\xymatrix{
S^3\wedge S^3\ar@/_9mm/[dd]_{1}\ar[d]^{1\wedge i}\ar[r]^{i\wedge 1}&\m{u}{3}\wedge S^3\ar[d]^{1\wedge i}\ar[r]^{q\wedge 1}&S^4\wedge S^3\\
S^3\wedge \m{r}{3}\ar@{=}[dr]\ar[r]^{i\wedge 1}&\m{u}{3}\wedge \m{r}{3}\ar[d]^{\tau'_{3}}&\\
S^3\wedge S^3\ar[r]^{1\wedge i}&S^3\wedge \m{r}{3}&\\
}$$
we get $$(q\wedge1)(1\wedge i\eta)=i\eta q\in [\m{u}{7},\m{r}{7}]; ~~~\tau'_{3}(1\wedge i)\in [\m{u}{6}, \m{r}{6}]^{2^{u-r}}_{1}.$$
which implies
$\tau'_{3}(1\wedge i\eta)\in [\m{u}{7}, \m{r}{6}]^{\eta}_{\eta}$ for $u=r$ and
  $\tau'_{3}(1\wedge i\eta)\in [\m{u}{7}, \m{r}{6}]^{0}_{\eta}$  for $u>r$. Since
$\eta\wedge 1: S^4\wedge \m{r}{3}\rightarrow S^3\wedge \m{r}{3}$ is also an element in $[\m{r}{7}, \m{r}{6}]^{\eta}_{\eta}$,

$$\tau'_{3}(1\wedge i\eta)=\left\{
  \begin{array}{ll}
    \eta\wedge 1+\kappa i\eta\eta q, & \hbox{$r=u>1$} \\
    \eta_{r}^{u}+\kappa i\eta\eta q, & \hbox{$u>r$}
  \end{array}
\right.~.$$

For $r>u\geq 1$, similarly, there is a retraction $\tau_3$ of $1\wedge i$ yielding the following commutative
diagram
     $$\xymatrix{
    &\m{u}{3}\wedge S^4\ar[d]^{1\wedge i\eta}\ar[dl]_{1\wedge \eta}\ar[dr]^{0}& \\
    \m{u}{3}\wedge S^3\ar[r]^{1\wedge i}&\m{u}{3}\wedge\m{r}{3}\ar@/^/[l]^{\tau_3}\ar[r]^{1\wedge q}& \m{u}{3} \wedge S^4
      }~$$
 and  for   $i\eta\wedge 1: S^4\wedge \m{r}{3}\rightarrow \m{u}{3}\wedge\m{r}{3}$  we have
  $$\tau_3(i\eta\wedge 1)=\eta_{u}^r+\kappa i\eta\eta q\in [\m{r}{7},\m{u}{6}];~~~(1\wedge q)(i\eta\wedge 1)=i\eta q\in [\m{r}{7},\m{u}{7}].$$
Note that  for $r>u$, the composition of $\m{r}{7}\xrightarrow{B(\chi)}\m{u}{7}\xrightarrow{1\wedge\eta}\m{u}{6}$ is an element in
 $[\m{r}{7}, \m{u}{6}]^{0}_{\eta}$, hence
  $$(1\wedge \eta) B(\chi)=\eta_{u}^{r}+\kappa'i\eta\eta q$$
and similarly, for the composition of  $\m{u}{7}\xrightarrow{B(\chi)}\m{r}{7}\xrightarrow{\eta\wedge 1}\m{r}{6}~(u>r)$, we have
 $$(\eta\wedge 1) B(\chi)=\eta_{r}^{u}+\kappa'i\eta\eta q.$$

From the calculations above
\begin{align}
\mathcal{B}_{1}= \left\{
     \begin{array}{ll}
      \footnotesize{ \begin{tabular}{c|cc|}
     \multicolumn{1}{c}{}  & $\m{r}{7}$ & \multicolumn{1}{c}{$\m{u}{7}$}\\
       \cline{2-3}
   $\m{u}{6}$ & $\eta_{u}^r+\kappa i\eta\eta q$  & $1\wedge \eta$\\
    $\m{u}{7}$ & $i\eta q$  & $0$\\
     \cline{2-3}
 \multicolumn{3}{c}{}\\
   \end{tabular}}\xrightarrow[\textbf{Tr1}]{\cong}\footnotesize{\begin{tabular}{c|cc|}
     \multicolumn{1}{c}{}  & $\m{r}{7}$ & \multicolumn{1}{c}{$\m{u}{7}$}\\
       \cline{2-3}
   $\m{u}{6}$ & $0$  & $1\wedge \eta$\\
    $\m{u}{7}$ & $i\eta q$  & $0$\\
     \cline{2-3}
 \multicolumn{3}{c}{}\\
   \end{tabular}} ~, & \hbox{$r>u\geq 1$} \\
         \footnotesize{\begin{tabular}{c|cc|}
     \multicolumn{1}{c}{}  & $\m{r}{7}$ & \multicolumn{1}{c}{$\m{u}{7}$}\\
       \cline{2-3}
   $\m{r}{6}$ & $\eta\wedge 1$  & $\eta\wedge 1+\kappa i\eta\eta q$\\
    $\m{r}{7}$ & $0$  & $i\eta q$\\
     \cline{2-3}
 \multicolumn{3}{c}{}\\
   \end{tabular}}\xrightarrow[\textbf{Tr2}]{\cong}\footnotesize{\begin{tabular}{c|cc|}
     \multicolumn{1}{c}{}  & $\m{r}{7}$ & \multicolumn{1}{c}{$\m{u}{7}$}\\
       \cline{2-3}
  $\m{r}{6}$ & $\eta\wedge 1$  & $0$\\
    $\m{r}{7}$ & $0$  & $i\eta q$\\
     \cline{2-3}
\multicolumn{3}{c}{}\\
   \end{tabular}}~, & \hbox{$r=u>1$} \\
       \footnotesize{\begin{tabular}{c|cc|}
     \multicolumn{1}{c}{}  & $\m{r}{7}$ & \multicolumn{1}{c}{$\m{u}{7}$}\\
       \cline{2-3}
   $\m{r}{6}$ & $\eta\wedge 1$  & $\eta_{r}^u+\kappa i\eta\eta q$\\
    $\m{r}{7}$ & $0$  & $i\eta q$\\
     \cline{2-3}
 \multicolumn{3}{c}{}\\
   \end{tabular}}\xrightarrow[\textbf{Tr3}]{\cong}\footnotesize{\begin{tabular}{c|cc|}
     \multicolumn{1}{c}{}  & $\m{r}{7}$ & \multicolumn{1}{c}{$\m{u}{7}$}\\
       \cline{2-3}
  $\m{r}{6}$ & $\eta\wedge 1$  & $0$\\
    $\m{r}{7}$ & $0$  & $i\eta q$\\
     \cline{2-3}
 \multicolumn{3}{c}{}\\
   \end{tabular}}~, & \hbox{$1\leq r< u$}\\
     \footnotesize{\begin{tabular}{c|cc|}
     \multicolumn{1}{c}{}  & $\moo{7}$ & \multicolumn{1}{c}{$\moo{7}$}\\
       \cline{2-3}
   $C_{1}^{8,1}$ & $1\wedge i\eta$  & $i\eta\wedge 1$\\
     \cline{2-3}
   \end{tabular}}~, & \hbox{$r=u=1$.} \\
     \end{array}
   \right.   \label{B1}
\end{align}

where the invertible transformations are given by
 \begin{itemize}
   \item [\textbf{Tr1}]:  $\ccc{2}(-B(\chi)-(\kappa+\kappa')i\eta q)+\ccc{1}$;
   \item [\textbf{Tr2}]:  $\ccc{1}(1+\kappa i\eta q)+\ccc{2}$;
   \item [\textbf{Tr3}]:  $\ccc{1}(-B(\chi)-(\kappa+\kappa')i\eta q)+\ccc{2}$.
 \end{itemize}

  \begin{itemize}
     \item  For $s\geq u$,  since $2^s=0\in [\m{u}{7},\m{u}{7}]$ for $s>1$ and take invertible transformation $\ccc{1}q+\ccc{3}$ on $\mathcal{B}$ for $s=1$, we get
    \begin{align}
\mathcal{B}\cong \footnotesize{\begin{tabular}{c|c|}
     \multicolumn{1}{c}{}   & \multicolumn{1}{c}{$S^{8}$}\\
       \cline{2-2}
    $\m{u}{7}$  & $i\eta $\\
     \cline{2-2}
   \end{tabular}}\oplus \mathcal{B}_1, \label{idecompB}
\end{align}
  So \begin{align}(C^{5}_u\wedge C_{r}^{5,s})^{(9)}\simeq \CH_{\mathcal{B}}\simeq \CH_{i\eta}\vee \CH_{\mathcal{B}_1}= C_{u}^9\vee \CH_{\mathcal{B}_1} \label{codimCC1}
\end{align}

    \begin{lemma}\label{lemma4.5}
  The mapping cone of the map  $\footnotesize{ \begin{tabular}{c|cc|}
     \multicolumn{1}{c}{}  & $\moo{7}$ & \multicolumn{1}{c}{$\moo{7}$}\\
       \cline{2-3}
   $C_{1}^{8,1}$ & $1\wedge i\eta$  & $i\eta\wedge 1$\\
     \cline{2-3}
   \end{tabular}}$ is homotopy equivalent to $\moo{3}\wedge C_{\eta}^5\vee C_{1}^{9,1}$.
\end{lemma}

\begin{proof}
$\Sigma C_{1}^{5}\wedge C_{1}^{5,1}\simeq C_{1}^5\wedge \moo{3}\wedge \moo{1}\simeq(\moo{3}\wedge C_{\eta}^5\vee\moo{7})\wedge\moo{1}\simeq C_{1}^{6,1}\wedge C_{\eta}^5\vee C_{1}^{10,1}$, hence $C_{1}^{5}\wedge C_{1}^{5,1}\simeq C_{1}^{5,1}\wedge C_{\eta}^5\vee C_{1}^{9,1}$.
Together with $(C_{1}^{5,1}\wedge C_{\eta}^5)^{(9)}\simeq \moo{3}\wedge C_{\eta}^5\vee C_{1}^9$ (from Lemma \ref{lemma4.3}), we have
$$(C_{1}^{5}\wedge C_{1}^{5,1})^{(9)}\simeq  \moo{3}\wedge C_{\eta}^5\vee C_{1}^9 \vee C_{1}^{9,1}.$$
 On the other hand, from (\ref{codimCC1}),  $$(C_{1}^{5}\wedge C_{1}^{5,1})^{(9)}\simeq C_{1}^9\vee C_{1}^{8,1}\cup_{(1\wedge i\eta, i\eta\wedge 1)}\CH(\moo{7}\vee\moo{7}) $$
So by Lemma \ref{lemma2.7},
$C_{1}^{8,1}\cup_{(1\wedge i\eta, i\eta\wedge 1)}\CH(\moo{7}\vee\moo{7})\simeq  \moo{3}\wedge C_{\eta}^5 \vee C_{1}^{9,1}$.
 \end{proof}

 Now, from (\ref{B1}), (\ref{codimCC1}) and Lemma \ref{lemma4.5} we get
 \begin{align}
&\text{if}~ s\geq u ~\text{and}~ r\geq u, ~~~&\text{then}~ (C^{5}_u\wedge C_{r}^{5,s})^{(9)}\simeq C_{u}^9\vee C_{u}^{9,r}\vee \m{u}{3}\wedge C_{\eta}^5; \label{codimCC2}\\
&\text{if}~s\geq u>r, ~~~&\text{then}~ (C^{5}_u\wedge C_{r}^{5,s})^{(9)}\simeq C_{u}^9\vee C_{r}^{9,u}\vee \m{r}{3}\wedge C_{\eta}^5. \label{codimCC3}
\end{align}

 \item For $u>s$, it is easy to get

if $u>s$ and $u\geq r$,  then $$\mathcal{B}\cong\footnotesize{\begin{tabular}{c|ccc|}
     \multicolumn{1}{c}{}  & $S^8$ & $\m{r}{7}$&\multicolumn{1}{c}{$\m{u}{7}$}\\
       \cline{2-4}
    $\m{u}{7}$ &  $i\eta$ & $0$  &$2^s$\\
   $\m{r}{6} $ & $0$   &$\eta\wedge 1$ &$0$\\
   $\m{r}{7} $ & $0$   &$0$ &$ i\eta q$\\
       \cline{2-4}
   \end{tabular}}~,$$

hence   \begin{align}
(C^{5}_u\wedge C_{r}^{5,s})^{(9)}\simeq \m{r}{3}\wedge C_{\eta}^5\vee Z', \label{codimCC4}
         \end{align}
where $Z'=(\m{u}{7}\vee \m{r}{7})\cup_{\mfour{i\eta}{2^s}{0}{i\eta q}}\CH(S^8\vee \m{u}{7});$

if $r> u>s,$ then
\begin{align}
 \mathcal{B}=\footnotesize{\begin{tabular}{c|ccc|}
     \multicolumn{1}{c}{}  & $S^8$ & $\m{r}{7}$&\multicolumn{1}{c}{$\m{u}{7}$}\\
       \cline{2-4}
    $\m{u}{7}$ &  $i\eta$ & $0$  &$2^s$\\
   $\m{u}{6} $ & 0& $\eta_{u}^{r}+\kappa i\eta\eta q$  &$1\wedge \eta$\\
   $\m{u}{7} $ & $0$   &$ i\eta q$ &0 \\
       \cline{2-4}
   \end{tabular}}~.\label{codimCC5}
\end{align}
 \end{itemize}

 The decomposability of  $C_{u}^5\wedge C_{r}^{5,s}$ can be obtained from the structure of $(C_{u}^5\wedge C_{r}^{5,s})/S^6$ and $(\CC)^{(9)}$ now.

 \begin{itemize}
   \item [(i)]For $s<u<r$, suppose $\CC$ is decomposable,  by Corollary \ref{corollary4.4} and (\ref{Pinch1}),
      $\CC\simeq X\vee C_{u}^{9,r}.$
     However, $\pi_{8}C_{u}^{9,r}\cong \z{r+1}$, which is not a direct summand of $\pi_8(\CC)\cong [C^{5,u},C_{r}^{5,s}]\cong \z{s+1}\oplus \z{r}$, hence $\CC$ is indecomposable for $s<u<r$.

   \item [(ii)]For $u<s, u<r$, suppose $\CC$ is decomposable,  by Corollary \ref{corollary4.4} and (\ref{codimCC2}),
      $\CC\simeq X\vee C_{u}^{9,r}.$
   There is a cofibre sequence,

      $S^7\vee S^8\vee \m{s}{7}\vee \m{s}{8} \xrightarrow{\mathcal{A}
}S^7\vee S^8\vee S^{7}\vee \m{s}{7}\rightarrow (\CC)/S^6\rightarrow
   S^8\vee S^9\vee \m{s}{8}\vee \m{s}{9} \xrightarrow{\Sigma\mathcal{A}}S^8\vee S^9\vee S^{8}\vee \m{s}{8}$, where $\mathcal{A}$ is the map   (\ref{Pinch4}). So,
 $$\ZH\oplus \z{s}\xrightarrow{[\Sigma\mathcal{A}, S^9]}\ZH\oplus \z{s}\oplus \zz\rightarrow [(\CC)/S^6, S^9]\rightarrow 0$$
 where  $[\Sigma\mathcal{A}, S^9]=\footnotesize{ \begin{tabular}{c|cc|}
     \multicolumn{1}{c}{}  &$\ZH$&\multicolumn{1}{c}{$\z{s}$}\\
       \cline{2-3}
 $ \ZH$ & $2^r$  & 0 \\
    $\z{s}$ &$0$ & $2^u$\\
   $\zz$ &1 & 1\\
     \cline{2-3}
   \end{tabular}}$ ~( $k: A_{1}\rightarrow A_{2}$ denotes the homomorphism of abelian groups defined by multiplication by $k$). Thus
  \begin{align}
   [\CC, S^9]\cong [(\CC)/S^6, S^9]\cong \frac{\ZH\oplus \z{s}\oplus\zz}{\lrg{(2^r,0,1), (0,2^u,1)}}  \nonumber\\
   \cong \frac{\z{r+1}\oplus\z{s}}{\lrg{(2^r,2^u)}}\cong\frac{\z{r+1}\oplus\z{s}}{\lrg{(0,2^u)}}\cong
    \z{r+1}\oplus\z{u}, \nonumber
  \end{align}
 which is a contradiction since $[C_{u}^{9,r}, S^9]\cong \z{r}$. Hence
   $\CC$ is indecomposable for $u<s$ and $u<r$.

 \item [(iii)] For $r\leq u<s$
   \begin{lemma}\label{lemma4.6}
    The wedge summand $$Z=(S^7\vee \m{s}{7})\cup_{\left(
                                   \footnotesize{\begin{array}{ccc}
                                     2^r & \eta q & 0 \\
                                     0 & 2^u & \eta\wedge 1 \\
                                   \end{array}}
                                 \right)}\CH(S^7\vee\m{s}{7}\vee\m{s}{8})$$ in $(\ref{Pinch3})$ is indecomposable.
    \end{lemma}
   \begin{proof}
     Assume that $Z\simeq Z_1\vee Z_2$ is decomposable and $H_{9}Z_1\neq 0$.
    From the mapping cone structure of $Z$, we get $[Z, S^9]\cong \frac{\z{s}\oplus \zz}{\lrg{(2^u,1)}}\cong \z{u+1}.$
   For the pinch map $P: \CC\rightarrow(\CC)/S^6$, there are isomorphisms
     $$P^{\ast}: \ch{\ast}{(\CC)/S^6}\xrightarrow{\cong} \ch{\ast}{\CC}~(\ast=7,8,9)$$
   Thus $Sq^2: \ch{7}{(\CC)/S^6}\rightarrow \ch{9}{(\CC)/S^6}$ is isomorphic
   and $Sq^2$ on $\ch{7}{Z}$ is also isomorphic which implies that Moore spaces can not be split out of $Z$.
    Together with $$H_{\ast}Z=\left\{
                             \begin{array}{ll}
                               \z{r}\oplus\z{u}, & \hbox{$\ast=7$} \\
                               \z{u}, & \hbox{$\ast=8$} \\
                               \z{s}, & \hbox{$\ast=9$} \\
                                0, & \hbox{otherwise}
                             \end{array}
                           \right.~,$$ we get
    $Z_2=C_{u}^{9,u}$ or $C_{r}^{9,u}$. Thus $[Z_2, S^9]\cong \z{u}$ which contradicts to $[Z, S^9]\cong\z{u+1}$.
   \end{proof}

Now assume $\CC$ is decomposable, by Corollary \ref{corollary4.4} and its homology,
  $$\CC\simeq X\vee C_{l}^{9,k},~~k,l\in\{u,r\}.$$
 Hence $C_{l}^{9,k}$ is split out of $(\CC)/S^6$, which is a contradiction since $(\CC)/S^6\simeq C^{9,r}\vee Z$ and $Z$ is indecomposable.
  Thus $\CC$ is indecomposable for $r\leq u<s$.

   \item [(iv)] For $u\geq s, u\geq r$
   \begin{lemma}\label{lemma4.7}
    The wedge summand $Z'$ in $(\ref{codimCC4})$ is homotopy equivalent to $C_{r}^{9,s}\vee C_{s}^{9}$, i.e.,
     $$(\m{u}{7}\vee \m{r}{7})\cup_{\mfour{i\eta}{2^s}{0}{i\eta q}}\CH(S^8\vee \m{u}{7})\simeq C_{r}^{9,s}\vee C_{s}^{9} ~\text{for}~u>s, u\geq r.$$
    \end{lemma}
   \begin{proof}
    Let $W:=(\m{u}{7}\vee \m{r}{7})\cup_{\left(
                                \begin{array}{cc}
                                  i\eta & 2^s i  \\
                                  0 & 0\\
                                \end{array}
                              \right)}\CH(S^8\vee S^7)=\m{r}{7}\vee U$, where
    $$U:=\m{u}{7}\cup_{(i\eta,2^si)}\CH(S^8\vee S^7)$$
   Since $U^{(8)}=\m{u}{7}\cup_{2^si}\CH S^7\simeq \m{s}{7}\vee S^8$, there is a cofibre sequence

  $S^{8}\xrightarrow{\mtwo{xi\eta}{a}}\m{s}{7}\vee S^8\rightarrow U$ for some $x\in \{0,1\}$ and $a\in \ZH$.
   By $H_{8}U=\ZH$ and $\pi_{8}U=\ZH$, we have $x=1$ and $a=0$. Hence
    \begin{align}
U\simeq S^8\vee C_{s}^9 ,~~ W\simeq\m{r}{7}\vee S^8\vee C_{s}^9. \label{U}
    \end{align}
   There is a commutative diagram
$$\xymatrix{
  S^8\vee S^{7} \ar[d]_{\mfour{1}{0}{0}{i}} \ar[r]^{\mfour{i\eta}{2^si}{0}{0}} & \m{u}{7}\vee \m{r}{7} \ar@{=}[d] \ar[r] & W \ar[d] \ar[r] &  S^9\vee S^{8}  \ar[d] \ar[r] & \m{u}{8}\vee \m{r}{8}\ar@{=}[d] \\
 S^8\vee \m{u}{7} \ar[r]_{\mfour{i\eta}{2^s}{0}{i\eta q}} &\m{u}{7}\vee \m{r}{7} \ar[r] & Z' \ar[r] & S^9\vee \m{u}{8} \ar[r]& \m{u}{8}\vee \m{r}{8} }$$
   \begin{align}
H_{\ast}(Z'/W)=\left\{
                     \begin{array}{ll}
                       0, & \hbox{otherwise} \\
                       \mathbb{Z}, & \hbox{$\ast =9$.}
                     \end{array}
                   \right.
   ~~\text{i.e.,}~~  Z'/W\simeq S^9. \label{Z'/W}
\end{align}
Thus there is a cofibre sequence
 $$S^8\xlongrightarrow{\mthree{yi\eta}{b}{0}}\m{r}{7}\vee S^8\vee C_{s}^9 \rightarrow Z'\rightarrow S^9.  $$ for some $y\in\{0,1\}$ and $b\in \ZH$.
From $H_{8}(Z')=\mathbb{Z}/2^s$,  $b=2^s$.

 Inclusion $J:(\CC)^{(9)}\rightarrow \CC$ induces isomorphisms
   $$J^{\ast}: \ch{\ast}{\CC}\xrightarrow{\cong}\ch{\ast}{(\CC)^{(9)}}~~\ast=7,8,9.$$
Thus $Sq^2$ is isomorphic on  both $\ch{7}{(\CC)^{(9)}}$ and $\ch{7}{Z'}$, which implies $y=1$. Hence
$Z'\simeq C_{r}^{9s}\vee  C_{s}^9.$
  \end{proof}

 From (\ref{Pinch2}), (\ref{codimCC3}), (\ref{codimCC4}), together with Lemma \ref{lemma4.7}, for $u\geq s, u\geq r$,
   \begin{align}
(\CC)/S^6\simeq  C_{r}^{9,s}\vee C^{9,r}\vee \m{s}{4}\wedge C_{\eta}^5; \label{Pinch u>=s,r}\\
(\CC)^{(9)}\simeq\m{r}{3}\wedge C_{\eta}^5\vee C_{s}^9\vee C_{r}^{9,s}. \label{CodimCC u>=s,r}
\end{align}

    By (\ref{CodimCC u>=s,r}), there is a cofibre sequence
    \begin{align}
 S^9\xrightarrow{\mthree{\alpha}{\beta}{\gamma}} \m{r}{3}\wedge C_{\eta}^5\vee C_{s}^9\vee C_{r}^{9,s}\rightarrow \CC. \label{Cofcodim1}
\end{align}
 From (\ref{Pi(2n+3)MCeta}), $\alpha=i_{\overline{C}}(t\varrho_6)$ for some $t\in \ZH$. Hence $(\CC)/S^6\simeq  C^{9,r}\vee (C_{s}^9\vee C_{r}^{9,s})\cup_{\mtwo{\beta}{\gamma}}\CH S^9$. By (\ref{Pinch u>=s,r}),
   $$(C_{s}^9\vee C_{r}^{9,s})\cup_{\mtwo{\beta}{\gamma}}\CH S^9\simeq \m{s}{4}\wedge C_{\eta}^5\vee  C_{r}^{9,s}=(C_{s}^9\vee C_{r}^{9,s})\cup_{\mtwo{h_{s}}{0}}\CH S^9.$$
    By the proof of Lemma \ref{lemma2.7}, there is a homotopy equivalence $\mu$ yielding following commutative diagram
  $$\xymatrix{
 S^9 \ar@{=}[d] \ar[r]^{\mtwo{\beta}{\gamma}~~} & C_{s}^9\vee C_{r}^{9,s} \ar[d]^{\mu'} \ar[r] & (C_{s}^9\vee C_{r}^{9,s})\cup_{\mtwo{\beta}{\gamma}}\CH S^9 \ar[d]^{\mu} \ar[r] &  S^{10}\ar@{=}[d] \\
S^9\ar[r]_{\mtwo{h_{s}}{0}~~} &C_{s}^9\vee C_{r}^{9,s} \ar[r] & (C_{s}^9\vee C_{r}^{9,s})\cup_{\mtwo{h_{s}}{0}}\CH S^9  \ar[r] & S^{10} }$$
  where $\mu'$ is the restriction of $\mu$ which is a self-homotopy equivalence of $C_{s}^9\vee C_{r}^{9,s} $.
   So we get the commutative diagram
    $$\xymatrix{
 S^9 \ar@{=}[d] \ar[r]^{\mthree{\alpha}{\beta}{\gamma}~~~~~~~~~~~~~} & \m{r}{3}\wedge C_{\eta}^5\vee C_{s}^9\vee C_{r}^{9,s} \ar[d]^{\Gamma} \ar[r] & \CC\ar@{.>}[d]^{\widetilde{\Gamma}} \\
S^9\ar[r]_{\mthree{\alpha}{h_s}{0}~~~~~~~~~~~~~} &\m{r}{3}\wedge C_{\eta}^5\vee C_{s}^9\vee C_{r}^{9,s} \ar[r] & \CH_{\mtwo{\alpha}{h_{s}}}\vee  C_{r}^{9,s}  }$$
   where $\Gamma =\footnotesize{\begin{tabular}{c|cc|}
     \multicolumn{1}{c}{}  &$\m{r}{3}\wedge C_{\eta}^5$&\multicolumn{1}{c}{$C_{s}^9\vee C_{r}^{9,s}$}\\
       \cline{2-3}
 $\m{r}{3}\wedge C_{\eta}^5$ & 1  & 0\\
    $C_{s}^9\vee C_{r}^{9,s}$ &$0$ & $\mu'$\\
     \cline{2-3}
   \end{tabular}}$ is a homotopy equivalence, which implies that $\widetilde{\Gamma}$ is also a homotopy equivalence. Thus
   $$\CC\simeq \mathbf{C}_{\mtwo{\alpha}{h_{s}}}\vee  C_{r}^{9,s},~~~~\alpha=i_{\overline{C}}(t\varrho_6). $$
   where $\mathbf{C}_{\mtwo{\alpha}{h_{s}}}$ is indecomposable ( Corollary \ref{corollary4.4})  which  implies that $t=1$ for $r=1$ and $t\in\{1,2\}$ for $r>1$ (By (\ref{Pi(2n+3)MCeta})).

   \begin{lemma}\label{lemma4.8}
  $$\mathbf{C}_{\mtwo{\alpha}{h_{s}}}\simeq C_{\eta}^5\wedge C_{r}^{5,s}.$$
    \end{lemma}
   \begin{proof}
 From Lemma \ref{lemma4.3}, it suffices to show that $t=1$, i.e., $\alpha=i_{\overline{C}}\varrho_6$.

For $r=1$, $t=1$.

For $r>1$,
by the similar computation of $\pi_{9}(C_{\eta}^5\wedge C_{r}^{5,s})$ in Lemma \ref{lemma4.3} we get
  \begin{align}
& \pi_{9}(\coH{\mtwo{\alpha}{h_s}})\cong \left\{
  \begin{array}{ll}
  \ZH/4\oplus \z{s}, & \hbox{$t=2$} \\
  \ZH/2\oplus \z{s+1}, & \hbox{$t=1$}
    \end{array}
  \right.. \label{pi9C(alpha hs)}
\end{align}
$t$ can be determined by computing
 \begin{align}
\pi_{9}(\CC)=\pi_{9}(\coH{\mtwo{\alpha}{h_s}})\oplus \pi_{9}(C_{r}^{9,s}), \label{11111}
\end{align}

 while $\pi_{9}(\CC)\cong [C^{14,u}, C^{13,s}_{r}]\cong [C^{7,u}, C^{6,s}_{r}].$
\textbf{Cof3} of $C^{7,u}$ yields the following exact sequence
  $$[S^7, C_{r}^{6,s}]\xrightarrow{(2^uq_{\eta})^{\ast}}[C_{\eta}^7, C_{r}^{6,s}]\rightarrow [C^{7,u}, C_{r}^{6,s}]\rightarrow [S^6, C_{r}^{6,s}]\xrightarrow{(2^uq_{\eta})^{\ast}}[C_{\eta}^6, C_{r}^{6,s}]$$
 From \textbf{Cof4} of $C_{r}^{k,s}$ and Lemma \ref{lemma3.1}, for $k\geq 6$
 \begin{align}
[S^{k+1}, C_{r}^{k,s}]\cong [S^{k+1}, \m{r}{k-2}]\oplus [S^{k+1}, S^k]\cong  \left\{
                                                             \begin{array}{ll}
                                                               \ZH/4\oplus \zz\oplus \zz, & \hbox{$r\geq 2$} \\
                                                                \ZH/2\oplus \zz\oplus \zz, & \hbox{$r=1$}
                                                             \end{array}
                                                           \right.. \label{[Sk+1,Crks]}
\end{align}
  By the fact $u\geq r\geq 2$ and from (\ref{pi9C(CetaCrs)}), there is a short exact sequence
  $$0\rightarrow \zz\oplus \z{s+1}\rightarrow [C^{7,u}, C_{r}^{6,s}]\rightarrow \zz\oplus \zz\rightarrow0.$$
Together with (\ref{11111}),  we get $t=1$.
\end{proof}

 From Lemma \ref{lemma4.8}, there is a decomposition
  $$\CC\simeq C_{r}^{9,s}\vee C_{\eta}^5\wedge C_{r}^{5,s}~~(u\geq s, u\geq r).$$

 \item [(v)] For $u=s<r$, from (\ref{Pinch1}) and (\ref{codimCC2}), we have
 \begin{align}
(\CC)/S^6\simeq  C_{r}^{9,s}\vee C^{9,r}\vee \m{s}{4}\wedge C_{\eta}^5; \nonumber \\
(\CC)^{(9)}\simeq\m{r}{3}\wedge C_{\eta}^5\vee C_{s}^9\vee C_{r}^{9,s}. \nonumber
\end{align}
   By the same proof as in the case (iv), we get that
   $$\CC\simeq C_{s}^{9,r}\vee C_{\eta}^5\wedge C_{s}^{5,s}~~(u=s<r).$$

   In summary   $\CC$ is
    \begin{itemize}
      \item [$\diamond$] homotopy equivalent to  $C_{r}^{9,s}\vee C_{\eta}^5\wedge C_{r}^{5,s}$ for $u\geq s, u\geq r$;
      \item [$\diamond$] homotopy equivalent to  $C_{s}^{9,r}\vee C_{\eta}^5\wedge C_{s}^{5,s}$ for $u=s<r$;
      \item [$\diamond$] indecomposable, otherwise.
    \end{itemize}

   By $C^{5,u}\wedge C_{r}^{5,s}\simeq D(C^{5}_{u}\wedge C_{s}^{5,r})$, we have $C^{5,u}\wedge C_{r}^{5,s}$ is
  \begin{itemize}
      \item [$\diamond$] homotopy equivalent to  $C_{r}^{9,s}\vee C_{\eta}^5\wedge C_{r}^{5,s}$ for $u\geq s, u\geq r$;
      \item [$\diamond$] homotopy equivalent to  $C_{s}^{9,r}\vee C_{\eta}^5\wedge C_{r}^{5,r}$ for $u=r<r$;
      \item [$\diamond$] indecomposable, otherwise.
    \end{itemize}
\end{itemize}

\section{Decomposition of $C_{r}^{5,s}\wedge C_{r'}^{5,s'}$, $r,r',s,s'\in \ZH^{+}$}
\label{sec5}
\subsection{Preliminaries}
\label{subset5.1}
In this section, let $u_3, u_4, \overline{u}_{4}, u_5$ (resp. $u'_3, u'_4, \overline{u'}_{4}, u'_5$) be generators of $\ch{\ast}{C_{r}^{5,s}}$ (resp. $\ch{\ast}{C_{r'}^{5,s'}}$) which satisfy conditions (\ref{Sq1}) of Lemma \ref{lemma3.3}.
$$H_{\ast}(\CCC)=\left\{      \begin{array}{ll}
    \z{m_{r,r'}}, & \hbox{$\ast=6$} \\
   \z{m_{r,s'}}\oplus \z{m_{s,r'}}\oplus  \z{m_{r,r'}}, & \hbox{$\ast=7$} \\
   \z{m_{s,s'}}\oplus \z{m_{r,s'}}\oplus  \z{m_{s,r'}}, & \hbox{$\ast=8$} \\
     \z{m_{s,s'}}, & \hbox{$\ast=9$} \\
        0, & \hbox{otherwise}
           \end{array}
   \right.$$

$$\ch{\ast}{\CCC}=\left\{
   \begin{array}{ll}
  \zz\{u_3\otimes u'_{3}\}, & \hbox{$\ast=6$} \\
  \zz\{u_3\otimes u'_{4}, u_3\otimes \overline{u}'_{4}, u_4\otimes u'_{3}, \overline{u}_4\otimes u'_{3}\}, & \hbox{$\ast=7$} \\
  \zz\left\{
     \begin{array}{c}
       u_4\otimes u'_{4}, u_4\otimes \overline{u}'_{4}, \overline{u}_4\otimes u'_{4},\\
      \overline{u}_4\otimes \overline{u}'_{4}, u_3\otimes u'_{5}, u_5\otimes u'_{3} \\
     \end{array}
   \right\}
, & \hbox{$\ast=8$} \\
\zz\{u_5\otimes u'_{4}, u_5\otimes \overline{u}'_{4}, u_4\otimes u'_{5}, \overline{u}_4\otimes u'_{5}\}, & \hbox{$\ast=9$} \\
\zz\{u_5\otimes u'_{5}\}, & \hbox{$\ast=10$} \\
0, & \hbox{otherwise}
    \end{array}
 \right.$$
The Steenrod operation action on $H^{\ast}(\CCC; \zz)$ is given as follows
\begin{itemize}
     \item [(i)] $Sq^4(u_{3}\otimes u'_{3})=u_{5}\otimes u'_{5}$;
     \item [(ii)] $Sq^2(u_{3}\otimes u'_{5})=Sq^2(u_{5}\otimes u'_{3})=u_{5}\otimes u'_{5}$;
     \item [(iii)]$Sq^2(u_{3}\otimes u'_{3})=\left\{
                                               \begin{array}{ll}
                                                 u_3\otimes u'_5+u_4\otimes u'_4\otimes +u_5\otimes u'_3, & \hbox{$r=r'=1$} \\
                                                  u_3\otimes u'_5+ u_5\otimes u'_3, & \hbox{otherwise}
                                               \end{array}
                                             \right.$;
     \item [(iv)]  $Sq^2(u_3\otimes u'_4)= u_5\otimes u'_4;  Sq^2(u_4\otimes u'_3)= u_4\otimes u'_5;$
  \newline  $Sq^2(u_3\otimes \overline{u}'_4)=\left\{
                                                \begin{array}{ll}
                                                  u_5\otimes\overline{u}'_4+u_4\otimes u'_{5}, & \hbox{$r=s'=1$} \\
                                                   u_5\otimes\overline{u}'_4, & \hbox{otherwise}
                                                \end{array}
                                              \right.;$
\newline $Sq^2(\overline{u}_4\otimes u'_3)= \left\{
                                                \begin{array}{ll}
                                                  u_5\otimes u'_4+\overline{u}_4\otimes u'_{5}, & \hbox{$r'=s=1$} \\
                                                   \overline{u}_4\otimes u'_{5}, & \hbox{otherwise}
                                                \end{array}
                                              \right.;$

 \end{itemize}

  \begin{lemma}\label{lemma5.1}
  If $\CCC$ is decomposable, then $\CCC\simeq X\vee C^{9,k}_{l}$ or $\CCC\simeq X\vee C^{9,k}_{l}\vee C^{9,k'}_{l'}$, where $X$ is indecomposable, $H_6{X}\neq 0$ and $\{k,k'\}\subset \{ m_{r,s'}, m_{s,r'}, m_{r,r'}\}$, $\{l,l'\}\subset \{ m_{s,s'}, m_{r,s'}, m_{s,r'}\}$.
  \end{lemma}
\begin{proof}
 Assume $\CCC\simeq X\vee Y$, $X$ is indecomposable and $H_{6}X\neq 0$. From Lemma \ref{lemma2.6}, $H_{t}X\cong H_{t}(\CCC), t=6,9$ and
 $dim~H_{7}X+dim~H_{8}X\geq 2$. It follows from the isomorphism $Sq^2: \ch{7}{\CCC}\xlongrightarrow{\cong } \ch{9}{\CCC}$ that Moore spaces are not summands of $Y$. Hence there will be following two cases
  \begin{itemize}
    \item [(i)]  $dim~H_{7}X+dim~H_{8}X=2$ which implies $dim~H_{7}X=dim~H_{8}X=1$. Note that $Y\in \mathbf{A}_{n}^2$. Thus $Y\simeq C^{9,k}_{l}\vee C^{9,k'}_{l'}$ for some $\{k,k'\}\subset \{ m_{r,s'}, m_{s,r'}, m_{r,r'}\}$, $\{l,l'\}\subset \{ m_{s,s'}, m_{r,s'}, m_{s,r'}\}$.
    \item [(ii)] $dim~H_{7}X+dim~H_{8}X=4$ which implies $dim~H_{7}X=dim~H_{8}X=2$ and  $Y\simeq C^{9,k}_{l}$ for some  $k\in \{ m_{r,s'}, m_{s,r'}, m_{r,r'}\}$, $l\in \{ m_{s,s'}, m_{r,s'}, m_{s,r'}\}$.
  \end{itemize}
\end{proof}

 Let $max=\max\{r,s,r',s'\}$. By $\CCC\simeq D(C^{5,r}_{s}\wedge C^{5,r'}_{s'})$, we can assume $max=s$.
\begin{lemma}\label{lemma5.2}
If $max=s>r', s'$, then $\CCC\simeq C^{9,s'}_{r'}\vee C_{r}^5\wedge C^{5,s'}_{r'}$, hence
 $$\CCC\simeq \left\{
               \begin{array}{ll}
                 C^{9,s'}_{r'}\vee  C^{9,s'}_{r'}\vee C_{\eta}^5\wedge C^{5,s'}_{r'}, & \hbox{$~~s>r',s'$ and $ r\geq r',s'$} \\
                 C^{9,s'}_{r'}\vee C^{9,r'}_{s'}\vee C_{\eta}^5\wedge C^{5,r}_{r}, & \hbox{$~~s>r'>s'=r$} \\
                 C^{9,s'}_{r'}\vee C_{r}^5\wedge C^{5,s'}_{r'}, & \hbox{$~~s>r',s'$ and $s'\neq r<r'$ or $r<s'$}
               \end{array}
             \right..$$
\end{lemma}
\begin{proof}
 From \textbf{Cof3} of $C_{r}^{5,s}$ and $|[C^{5,s'}_{r'}, C^{5,s'}_{r'}]|=max\{2^{s'+1},2^{r'+1}\}$ \cite{RefJ.H.C}, we get
   $$\xymatrix{
S^4 \wedge C^{5,s'}_{r'}\ar[drr]_{\mtwo{i\eta\wedge 1}{2^s\wedge 1=0}~~~~~}\ar[rr]^{\mtwo{i\eta}{2^s}\wedge 1}&& (\m{r}{3}\vee S^3)\wedge C^{5,s'}_{r'}\ar[r] &\CCC \\
&&\m{r}{3}\vee C^{5,s'}_{r'}\wedge S^4\vee C^{5,s'}_{r'}\ar[u]^{\simeq}&
}$$
Thus $\CCC\simeq C^{5,s'}_{r'}\vee (\m{r}{3}\vee C^{5,s'}_{r'})\cup_{i\eta\wedge 1}\CH (S^4\vee C^{5,s'}_{r'})\simeq C^{5,s'}_{r'}\vee C_{r}^5\wedge C^{5,s'}_{r'}$. By the decomposability of $C_{r}^5\wedge C^{5,s'}_{r'}$ in Subsection \ref{subsec4.3}, the Lemma is obtained.
 \end{proof}
Now the following cases remain:
\begin{itemize}
  \item [(I)] $max=s=r'$

    (i)$s=r'>s'>r$;  (ii)$s=r'>s'=r$;  (iii)$s=r'=s'>r$;  (iv)$s=r'=s'=r$;  (v)$s=r'>r>s'$;  (vi)$s=r'=r>s'$;
  \item [(II)] $max=s=s'$

    (i)$s=s'>r'>r$;  (ii)$s=s'>r>r'$;  (iii)$s=s'>r'=r$; (iv)$s=s'=r>r'$.
\end{itemize}

By $\CCC\simeq D(C_{s}^{5,r}\wedge C_{s'}^{5,r'})\simeq D(C_{s'}^{5,r'}\wedge C_{s}^{5,r} )$,
 case (I)(iii) is dual to case (I)(vi); case (I)(i) is dual to case (I)(v).

By $\CCC\simeq C_{r'}^{5,s'}\wedge C_{r}^{5,s}$,
case (II)(i) is the same as the case (II)(ii); case (II)(iv) is is the same as the case (I)(iii).

Hence it suffices to compute the following cases, denoted by \textbf{Cases~\maltese} :

\begin{itemize}
  \item [] (a) $s=r'>s'>r$; (b) $s=r'>s'=r$; (c) $s=r'=s'>r$;
  \item [] (d) $s=r'=s'=r$; (e) $s=s'>r'>r$; (f) $s=s'>r'=r$.
\end{itemize}

We will prove the (a) of \textbf{Cases~\maltese}  and omit the proofs of other cases since they are similar or easier.

\subsection{$(\CCC)/S^6$  and $(\CCC)^{(9)}$ for $s=r'>s'>r$}
\label{subset5.2}

\textbf{(1) Determining $(\CCC)/S^6$}

 For  $S_1^3\vee \m{s}{3}\xlongrightarrow{f=(2^r, \eta q)}S_{a}^3\rightarrow C_{r}^{5,s}$~ and ~$S_2^3\vee \m{s'}{3}\xlongrightarrow{f'=(2^{r'}, \eta q)}S_{b}^3\rightarrow C_{r'}^{5,s'}.$
  $$(\CCC)/S^6\simeq (\Sigma(S_1^3\vee \m{s}{3})\wedge S_{b}^3\vee \Sigma S_{a}^{3}\wedge (S_2^3\vee \m{s'}{3}))\cup_{\mathcal{A}}\CH\Sigma(S_1^3\vee \m{s}{3})\wedge (S_2^3\vee \m{s'}{3}),$$
where $\mathcal{A}=\mtwo{\Sigma 1\wedge f'}{-\Sigma f\wedge 1}$, i.e.,
$$(\CCC)/S^6\simeq(S^7\vee \m{s}{7}\vee S^7\vee \m{s'}{7})\cup_{\mathcal{A}}\CH (S^7\vee\m{s'}{7}\vee\m{s}{7}\vee \Sigma\m{s}{3}\wedge\m{s'}{3}).$$
where $\mathcal{A}=~~~~~~~~~~~\footnotesize{\begin{tabular}{c|cccc|}
     \multicolumn{1}{c}{}  & $S^7$ &$ \m{s'}{7}$ &$ \m{s}{7}$&\multicolumn{1}{c}{$\Sigma\m{s}{3}\wedge\m{s'}{3}$}\\
       \cline{2-5}
   $~~~~~~~~~~~~~~~~~~S^7$ & $2^{r'}$  & $\eta q$ &$0$&0\\
    $\Sigma\m{s}{3}\wedge S^3_{b}=\m{s}{7}$ & $0$  & $0$ &$2^{r'}$&$\Sigma 1\wedge \eta q$\\
   $~~~~~~~~~~~~~~~~~~S^7$ & $-2^{r}$  & $0$  &$\eta q$&0\\
   $\Sigma S_{a}^3\wedge\m{s'}{3}=\m{s'}{7}$ & $0$  &$-2^{r}$ &0&$\Sigma \eta q \wedge 1$\\
     \cline{2-5}
    \multicolumn{5}{c}{} \\
   \end{tabular}}$~.

With identification $\Sigma\m{s}{3}\wedge\m{s'}{3}\simeq \m{s'}{7}\vee\m{s'}{8}$,
 $$\footnotesize{\mathcal{A}=\begin{tabular}{c|ccccc|}
     \multicolumn{1}{c}{}  & $S^7$ &$ \m{s'}{7}$ &$ \m{s}{7}$&$\m{s'}{7}$&\multicolumn{1}{c}{$\m{s'}{8}$}\\
       \cline{2-6}
   $S^7$ & $2^{r'}$  & $\eta q$ &$0$&0&0\\
    $\m{s}{7}$ & $0$  & $0$ &0&$i\eta q$&$\xi_{s}^{s'}+\kappa i\eta\eta q$\\
   $S^7$ & $-2^{r}$  & $0$  &$\eta q$&0&0\\
   $\m{s'}{7}$ & $0$  &$-2^{r}$ &0&0&$\eta\wedge 1$\\
     \cline{2-6}
   \end{tabular}}$$
$$\xrightarrow[\begin{array}{c}
                 2^{r'-r}\rrr{3}+\rrr{1}; \\
                 -\rrr{3};~-\rrr{4}
               \end{array}
]{\cong}\footnotesize{\begin{tabular}{c|ccccc|}
     \multicolumn{1}{c}{}  & $S^7$ &$ \m{s'}{7}$ &$ \m{s}{7}$&$\m{s'}{7}$&\multicolumn{1}{c}{$\m{s'}{8}$}\\
       \cline{2-6}
   $S^7$ & $0$  & $\eta q$ &$0$&0&0\\
    $\m{s}{7}$ & $0$  & $0$ &0&$i\eta q$&$\xi_{s}^{s'}+\kappa i\eta\eta q$\\
   $S^7$ & $2^{r}$  & $0$  &$\eta q$&0&0\\
   $\m{s'}{7}$ & $0$  &$2^{r}$ &0&0&$\eta\wedge 1$\\
     \cline{2-6}
   \end{tabular}}~,$$
Thus $(\CCC)/S^6\simeq C_{r}^{9,s}\vee L$, where $L$ is the mapping cone of the map
$$\begin{tabular}{c|ccc|}
     \multicolumn{1}{c}{}  & $ \m{s'}{7}$ &$\m{s'}{7}$&\multicolumn{1}{c}{$\m{s'}{8}$}\\
       \cline{2-4}
   $S^7$ & $\eta q$ &$0$&0\\
    $\m{s}{7}$ &0&$i\eta q$&$\xi_{s}^{s'}+\kappa i\eta\eta q$\\
   $\m{s'}{7}$ &$2^{r}$ &0&$\eta\wedge 1$\\
     \cline{2-4}
   \end{tabular}~.$$

\begin{lemma}\label{lemma5.3}
 $L\simeq C^{9,r}\vee (C_{s}^{9,s'}\vee C_{r}^9)\cup_{\mtwo{\alpha}{\gamma}}\CH S^9$
  where $\alpha=i_{\underline{M}}\mtwo{\xi_{s}}{0}$ and $\gamma$ is determined by $q_{S}\gamma=\mtwo{0}{2^{s'}}$, i.e., there are commutative diagrams
 $$\xymatrix{
                &         S^9 \ar[dl]_{\mtwo{\xi_{s}}{0}}\ar[d]^{\alpha}     \\
 \m{s}{7}\vee S^8 \ar[r]_{i_{\underline{M}}} & C_{s}^{9,s'}~;}~~~~~~\xymatrix{
 S^9 \ar[d]_{\gamma} \ar[dr]^{\mtwo{0}{2^{s'}}}        \\
 C_{r}^{9} \ar[r]_{q_{S}~~~}  & S^8\vee S^9~~.             }$$
\end{lemma}

\begin{proof}
By the compositions  $$\xymatrix{
          S^8  \ar[r]^{i}   \ar@/_4mm/[rr]_{i\eta}  & \m{s'}{8} \ar[r]^{\eta\wedge 1}  & \m{s'}{7} }~~~\text{and}~~
              \xymatrix{
          S^8  \ar[r]^{i}   \ar@/_4mm/[rrr]_{0}  & \m{s'}{8} \ar[rr]^{\xi_{s}^{s'}+\kappa i\eta\eta q}  && \m{s}{7} }~~,
$$
$L^{(9)}$ is the mapping cone of the map $\footnotesize{\begin{tabular}{c|ccc|}
     \multicolumn{1}{c}{}  & $ \m{s'}{7}$ &$\m{s'}{7}$&\multicolumn{1}{c}{$S^{8}$}\\
       \cline{2-4}
   $S^7$ & $\eta q$ &$0$&0\\
    $\m{s}{7}$ &0&$i\eta q$&0\\
   $\m{s'}{7}$ &$2^{r}$ &0&$i\eta$\\
     \cline{2-4}
   \end{tabular}}~.$ Hence $L^{(9)}\simeq C_{s}^{9,s'}\vee L_1$, where $L_1$ is the mapping cone of   $\footnotesize{\begin{tabular}{c|cc|}
     \multicolumn{1}{c}{}  & $ \m{s'}{7}$ &\multicolumn{1}{c}{$S^{8}$}\\
       \cline{2-3}
   $S^7$ & $\eta q$ &$0$\\
   $\m{s'}{7}$ &$2^{r}$ &$i\eta$\\
     \cline{2-3}
   \end{tabular}}~.$
   Let $W_1$ be the mapping cone of the map $\footnotesize{\begin{tabular}{c|cc|}
     \multicolumn{1}{c}{}  & $ S^{7}$ &\multicolumn{1}{c}{$S^{8}$}\\
       \cline{2-3}
   $S^7$ & 0 &$0$\\
   $\m{s'}{7}$ &$2^{r}i_{7}$ &$i\eta$\\
     \cline{2-3}
   \end{tabular}}~,$ i.e.,  $W_1\simeq (S^7\vee \m{s'}{7})\cup_{\mfour{0~~}{~~0}{2^r i_{7}}{i\eta}}\CH (S^7\vee S^8)\simeq S^7\vee \m{s'}{7}\cup_{(2^r i, i\eta)}\CH (S^7\vee S^8)\simeq S^7\vee S^8\vee C_{r}^9$ ~ (by (\ref{U})).

Similarly as in the proof of Lemma \ref{lemma4.7}, there is a cofibre sequence
 $$S^8\xlongrightarrow{\mthree{x\eta}{\hat{\beta}}{\hat{\gamma}}} W_1\simeq S^7\vee S^8\vee C_{r}^9\rightarrow L_1\rightarrow S^9,$$
where $x\in\{0,1\}$. From $H_{8}L_1=\z{r}$ and $\pi_{8}C_{r}^9=0$ we get $\hat{\beta}=2^r$ and $\hat{\gamma}=0$. By the isomorphisms
$\ch{7}{(\CCC)/S^6}\xrightarrow{\cong}\ch{7}{\CCC}$ and $\ch{9}{\CCC}\xrightarrow{\cong}\ch{9}{(\CCC)^{(9)}}$  induced by canonical pinch map and canonical inclusion respectively, $Sq^2$ on $\ch{7}{L}$ and  $\ch{7}{L_1}$ are also isomorphic, which implies that $x=1$. So $L_1\simeq C^{9,r}\vee C_{r}^9$ and $L^{(9)}\simeq C_{s}^{9,s'}\vee C^{9,r}\vee C_{r}^{9}$.
There is a cofibre sequence
 $$S^9\xrightarrow{f_L=\mthree{\alpha}{\beta}{\gamma}}C_{s}^{9,s'}\vee C^{9,r}\vee C_{r}^{9}\rightarrow L\rightarrow S^{10}\rightarrow C_{s}^{10,s'}\vee C^{10,r}\vee C_{r}^{10}.$$
$\alpha=i_{\underline{M}}\mtwo{z\xi_{s}}{x\eta}$, $\beta=i_{S}\mtwo{0}{y\eta}$, $\gamma$ is determined by $q_{S}\gamma=\mtwo{w\eta}{a}$, where $x,y,z,w\in\{0,1\}$ and $a\in\ZH$.
$$\xymatrix{
                & S^9 \ar[dl]_{\mtwo{z\xi_{s}}{x\eta}}\ar[d]^{\alpha}     \\
 \m{s}{7}\vee S^8 \ar[r]_{i_{\underline{M}}} & C_{s}^{9,s'}~;}
~~~~~\xymatrix{
                & S^9 \ar[dl]_{\mtwo{0}{y\eta}}\ar[d]^{\beta}     \\
 S^{7}\vee S^8 \ar[r]_{i_{S}} & C^{9,r}~;}
 ~~~~~~\xymatrix{
 S^9 \ar[d]_{\gamma} \ar[dr]^{\mtwo{w\eta}{a}}        \\
 C_{r}^{9} \ar[r]_{q_{S}~~~}  & S^8\vee S^9~~.             }$$
$a=2^{s'}$ from $[L, S^{10}]\cong \z{s'}$. From Proposition \ref{proposition 2.1}, for $s=r'>s'>r$
$$[(\CCC)/S^6, S^8]\cong [\CCC, S^8]\cong [C^{5,s}_{r}, C^{5,r'}_{s'}]\cong \z{s}\oplus\z{r+1}\oplus\z{r+1}.$$ Together with
$[(\CCC)/S^6, S^8]\cong [C^{9,s}_{r}, S^8]\oplus [L, S^8]$, we get $$[L, S^8]\cong \z{s}\oplus\z{r+1}.$$
On the other hand, by the cofibre sequence of $L$ above, there is an exact sequence
  $$0\rightarrow \frac{\zz}{\lrg{x,y,z,w}}\rightarrow [L, S^8]\rightarrow \z{s+1}\oplus\z{r+1}\xlongrightarrow{(\alpha^{\ast}, \gamma^{\ast})} \zz$$
where $\alpha^{\ast}(1)=z, \gamma^{\ast}(1)=w$. Hence $w=0, z=1$.
By the following commutative diagram
 $$\xymatrix{
   S^9 \ar@{=}[d] \ar[rr]^{f_L~~~~~~~~} & &C_{s}^{9,s'}\vee C^{9,r}\vee C_{r}^{9} \ar[d]_{\Theta} \ar[r] & L \ar@{.>}[d]^{\hat{\Theta}}\\
   S^9 \ar[rr]^{\theta=\Theta f_L~~~~~~~~} && C_{s}^{9,s'}\vee C^{9,r}\vee C_{r}^{9} \ar[r] & \coH{\theta}   }$$
where $\Theta=\footnotesize{\begin{tabular}{c|ccc|}
     \multicolumn{1}{c}{}  & $ C_{s}^{9,s'}$ &$C^{9,r}$&\multicolumn{1}{c}{$C_{r}^{9}$}\\
       \cline{2-4}
   $C_{s}^{9,s'}$ & $\hat{\mu}$ &$0$&0\\
    $C^{9,r}$ &$\hat{\lambda}$ &1&0\\
   $C_{r}^{9}$ &0 &0&1\\
     \cline{2-4}
   \end{tabular}}$ is a homotopy equivalence, $\hat{\mu}$ and $\hat{\lambda}$ are induced by the following commutative diagrams
   $$\xymatrix{
   S^8 \ar@{=}[d] \ar[rr]^{\mtwo{i\eta}{2^{s'}}~~~~~~~~} & &\m{s}{7}\vee S^8 \ar[d]_{\mfour{1}{0}{xq}{1}} \ar[r]^{i_{\underline{M}}} & C_{s}^{9,s'} \ar@{.>}[d]^{\hat{\mu}}\\
   S^8 \ar[rr]_{\mtwo{i\eta}{2^{s'}}~~~~~~~~} &&  \m{s}{7}\vee S^8 \ar[r]_{i_{\underline{M}}} & C_{s}^{9,s'}~,}~~~~~~~~~~\xymatrix{
   S^8 \ar[d]^{0} \ar[rr]^{\mtwo{i\eta}{2^{s'}}~~~~~~~~} & &\m{s}{7}\vee S^8 \ar[d]_{\mfour{0}{0}{yq}{0}} \ar[r]^{i_{\underline{M}}} & C_{s}^{9,s'} \ar@{.>}[d]^{\hat{\lambda}}\\
   S^8 \ar[rr]_{\mtwo{\eta}{2^{r}}~~~~~~~~} &&  S^{7}\vee S^8 \ar[r]_{i_{S}} & C^{9,r}~.}$$

Then $\theta=\Theta f_L=\mthree{\hat{\mu}\alpha}{\hat{\lambda}\alpha+\beta}{\gamma}=\mthree{\hat{\mu}\alpha}{0}{\gamma}$. Rewrite $\hat{\mu}\alpha$ as $\alpha$,
 $$L\simeq \coH{\theta}\simeq C^{9,r}\vee (C_{s}^{9,s'}\vee C_{r}^9)\cup_{\mtwo{\alpha}{\gamma}}\CH S^9$$
$\alpha, \gamma$ satisfy the conditions in the Lemma.
 \end{proof}
Thus
\begin{align}
&(\CCC)/S^6\simeq C_{r}^{9,s}\vee C^{9,r}\vee (C_{s}^{9,s'}\vee C_{r}^9)\cup_{\mtwo{\alpha}{\gamma}}\CH S^9,~ s=r'>s'>r.  \label{pinch s=r'>s'>r}
\end{align}

\textbf{(2) Determining $(\CCC)^{(9)}$}

$$S_1^4\xrightarrow{g=\mtwo{i\eta}{2^s}}\m{r}{3}\vee S_{a}^4\rightarrow C_{r}^{5,s};~~S_2^4\xrightarrow{g'=\mtwo{i\eta}{2^{s'}}}\m{r'}{3}\vee S_{b}^4\rightarrow C_{r'}^{5,s'},$$
$$(\CCC)^{(9)}\simeq (\m{r}{3}\vee S_{a}^4)\wedge (\m{r'}{3}\vee S_{b}^4)\cup_{\mathcal{B}}\CH (S_1^4\wedge(\m{r'}{3}\vee S_{b}^4)\vee(\m{r}{3}\vee S_{a}^4)\wedge S_2^4 ),$$
where $\mathcal{B}=(g\wedge 1, 1\wedge g')$, i.e.,
$$(\CCC)^{(9)}\simeq (\m{r}{3}\wedge\m{r'}{3}\vee \m{r}{7}\vee\m{r'}{7}\vee S^8)\cup_{\mathcal{B}}\CH (\m{r}{7}\vee S^8\vee \m{r'}{7}\vee S^8 ),$$
$$\footnotesize{\mathcal{B}=\begin{tabular}{c|cccc|}
     \multicolumn{1}{c}{}  & $\m{r}{7}$ &$S^8$ &$ \m{r'}{7}$&\multicolumn{1}{c}{$S^8$}\\
       \cline{2-5}
   $\m{r}{3}\wedge\m{r'}{3}$ & $1\wedge i\eta$  &0 &$i\eta\wedge1$&0\\
    $\m{r}{7}$ & $2^{s'}$  & $0$ &0&$i\eta $\\
   $\m{r'}{7}$ &0& $i\eta$  &$2^s$&0\\
   $S^8$ & $0$  &$2^{s'}$ &0&$2^s$\\
     \cline{2-5}
     \multicolumn{5}{c}{}\\
   \end{tabular}}\cong  \footnotesize{ \begin{tabular}{c|cccc|}
     \multicolumn{1}{c}{}  & $\m{r}{7}$ &$S^8$ &$ \m{r'}{7}$&\multicolumn{1}{c}{$S^8$}\\
       \cline{2-5}
   $\m{r}{3}\wedge\m{r'}{3}$ & $1\wedge i\eta$  &0 &$i\eta\wedge1$&0\\
    $\m{r}{7}$ & $0$  & $0$ &0&$i\eta $\\
   $\m{r'}{7}$ &0& $i\eta$  &$0$&0\\
   $S^8$ & $0$  &$2^{s'}$ &0&0\\
  \cline{2-5}
     \multicolumn{5}{c}{}\\
\end{tabular}}$$
by noting that $2^{s'}=0\in[\m{r}{7}, \m{r}{7}]$  and $2^{s}=0\in[\m{r'}{7}, \m{r'}{7}]$ for  $s=r'>s'>r$.
From (\ref{B1}), for $r'> r$,
$(\m{r}{3}\wedge\m{r'}{3})\cup_{(1\wedge i\eta, i\eta\wedge 1)}\CH (\m{r}{7}\vee\m{r'}{7})\simeq \m{r}{3}\wedge C_{\eta}^5\vee C_{r}^{9,r'}.$
Thus
\begin{align}
(\CCC)^{(9)}\simeq  \m{r}{3}\wedge C_{\eta}^5\vee C_{r}^{9,r'}\vee C_{r'}^{9,s'}\vee C_{r}^9 ~~\text{for}~ s=r'>s'>r \label{Codim Crs Cr's'}
 \end{align}

\subsection{Decomposition of \CCC for $s=r'>s'>r$}
\label{subsec5.3}

Denote  column vector by $(\zeta_1, \zeta_2, \cdots, \zeta_s)^{T}$ in  the  rest of the paper.
\\

From (\ref{Codim Crs Cr's'}), there is a cofibre sequence
$$S^9\xlongrightarrow{(\delta_1, \mu, \nu, \omega)^T}\m{r}{3}\wedge C_{\eta}^5\vee C_{r}^{9,s}\vee C_{r'}^{9,s'}\vee C_{r}^9\rightarrow \CCC$$
Since $\delta_1=i_{\overline{C}}t_1\varrho_6$ for some $t_1\in \ZH$,
   $$(\CCC)/S^6\simeq C^{9,r}\vee (C_{r}^{9,s}\vee C_{r'}^{9,s'}\vee C_{r}^9)\cup_{(\mu, \nu, \omega)^T}\CH S^9.$$

On the other hand, from (\ref{pinch s=r'>s'>r})
   $$(\CCC)/S^6\simeq C^{9,r}\vee (C_{r}^{9,s}\vee C_{r'}^{9,s'}\vee C_{r}^9)\cup_{(0,\alpha, \gamma)^T}\CH S^9.$$
Thus $(C_{r}^{9,s}\vee C_{r'}^{9,s'}\vee C_{r}^9)\cup_{(\mu, \nu, \omega)^T}\CH S^9\simeq (C_{r}^{9,s}\vee C_{r'}^{9,s'}\vee C_{r}^9)\cup_{(0,\alpha, \gamma)^T}\CH S^9$, which are 2-local spaces. Consequently, there is a homotopy equivalence $\lambda$ yielding the following homotopy commutative diagram
  $$\xymatrix{
   S^9 \ar@{=}[d] \ar[rr]^{(\mu, \nu, \omega)^T~~~~~~~~} & &C_{r}^{9,s}\vee C_{r'}^{9,s'}\vee C_{r}^{9} \ar[d]_{\lambda'} \ar[r] &\coH{(\mu, \nu, \omega)^T} \ar[d]^{\lambda~ \simeq}\ar[r] & S^{10}\ar@{=}[d]\\
   S^9 \ar[rr]^{(0, \alpha, \gamma)^T~~~~~~~~} && C_{r}^{9,s}\vee C_{r'}^{9,s'}\vee C_{r}^{9} \ar[r] & \coH{(0, \alpha, \gamma)^T}\ar[r] & S^{10}   }$$
where $\lambda'$ is the restriction of $\lambda$, which is a self-homotopy equivalence of $C_{r}^{9,s}\vee C_{r'}^{9,s'}\vee C_{r}^{9}$.

Let $\widetilde{\Gamma}$ be induced by the following commutative diagram
 $$\xymatrix{
   S^9 \ar@{=}[d] \ar[rr]^{(\delta_1,  \mu, \nu, \omega)^T~~~~~~~~~~~~~} & &\m{r}{3}\wedge C_{\eta}^5\vee C_{r}^{9,s}\vee C_{r'}^{9,s'}\vee C_{r}^{9} \ar[d]_{\Gamma} \ar[r] &\CCC \ar@{.>}[d]^{\widetilde{\Gamma}} \\
   S^9 \ar[rr]^{\theta=\Gamma(\delta_1,  \mu, \nu, \omega)^T ~~~~~~~~~~~~~~~~} &&\m{r}{3}\wedge C_{\eta}^5\vee C_{r}^{9,s}\vee C_{r'}^{9,s'}\vee C_{r}^{9} \ar[r] & \coH{\theta}   }$$
where $\Gamma=\footnotesize{\begin{tabular}{c|cc|}
     \multicolumn{1}{c}{} &$\m{r}{3}\wedge C_{\eta}^5$&\multicolumn{1}{c}{$ C_{r}^{9,s}\vee C_{r'}^{9,s'}\vee C_{r}^{9} $}\\
       \cline{2-3}
    $\m{r}{3}\wedge C_{\eta}^5$&1 & 0\\
   $ C_{r}^{9,s}\vee C_{r'}^{9,s'}\vee C_{r}^{9} $&0 & $\lambda'$\\
      \cline{2-3}
\end{tabular}}$ is a self-homotopy equivalence, then $\widetilde{\Gamma}$ is also a homotopy equivalence.
$$\theta=\Gamma(\delta_1,  \mu, \nu, \omega)^T=(\delta_1,\lambda'( \mu, \nu, \omega)^T )^T=(\delta_1, 0, \alpha, \gamma)^T.$$
Consequently, for $s=r'>s'>r$,
$$\CCC\simeq \coH{\theta}\simeq C_{r}^{9,s}\vee Q_1,~~ Q_1=(\m{r}{3}\wedge C_{\eta}^5\vee C_{r'}^{9,s'}\vee C_{r}^{9})\cup_{(\delta_1, \alpha, \gamma)^T}\CH S^9.$$

\begin{lemma}\label{lemma5.4}
$Q_1\simeq C^{5,s'}\wedge C_{r}^{5,s}$,~~~~$s=r'>s'>r$.
\end{lemma}

\begin{proof}
\textbf{(1) Determining  $\delta_{1}$, i.e., $t_{1}$}
\\

By Lemma \ref{lemma5.1}, $\CCC$ can not split out $\m{r}{3}\wedge C_{\eta}^5$, implies that  $\delta_1\neq 0$ in
 $\pi_{9}(\m{r}{3}\wedge C_{\eta}^5)$.
Hence we can assume that $t_1=1$ for $r=1$ and $t_1\in\{1,2\}$ for $r\geq 2$. Next we determine $t_1$ for $r\geq 2$.

\begin{lemma}\label{lemma5.5}
For $s\geq r', s'$ and $r\geq 2$, there is a short exact sequence of abelian groups
 $$0\rightarrow \zz\oplus \zz \oplus \zz\oplus\z{s'+\delta_{r'}} \rightarrow\pi_9(\CCC) \rightarrow \zz\oplus \zz \rightarrow 0.$$
where $\delta_{r'}=0$ for $r'=1$  and $\delta_{r'}=1$ for $r'>1$.
\end{lemma}

\begin{proof}
We only prove the case  $r'>1$.

By Proposition \ref{proposition 2.1},
$$\pi_{9}(\CCC)\cong [C_{s}^{5,r}\wedge C_{s'}^{5,r'}, S^7]\cong [C_{s}^{7,r}, C_{r'}^{6,s'}].$$
By \textbf{Cof5} of $C_{s}^{7,r}$, there is an exact sequence
\begin{align}
[S^7, C_{r'}^{6,s'}]\xlongrightarrow{(2^rp_1q_S)^{\ast}} [C_{s}^7, C_{r'}^{6,s'}]\rightarrow [C_{s}^{7,r}, C_{r'}^{6,s'}]\rightarrow [S^6, C_{r'}^{6,s'}]\xlongrightarrow{0} [C_{s}^7, C_{r'}^{6,s'}] . \label{[Cs7,C_r'6,s']}
\end{align}
From (\ref{[Sk+1,Crks]}), $(2^rp_1q_S)^{\ast}$ is zero for $r\geq 2$. For $s\geq r', s'$,
$[C_{s}^7, C_{r'}^{6,s'}]\cong [S^9, C^{5,s}\wedge C_{r'}^{5,s'}]\cong \pi_{9}(C_{r'}^{9,s'}\vee C_{\eta}^5\wedge C_{r'}^{5,s'})\cong \zz\oplus \zz\oplus \zz\oplus \z{s'+1}$ (the last isomorphism is from (\ref{pi9C(CetaCrs)})).
\end{proof}

There is a cofibre sequence
\begin{align}
S^9\xlongrightarrow{(\delta_{1}, \alpha, \gamma)^T}\m{r}{3}\wedge C_{\eta}^5\vee C_{r'}^{9,s'}\vee C_{r}^{9}\rightarrow Q_1\rightarrow S^{10}. \label{CofQ1}
\end{align}
which implies following exact sequence
\begin{footnotesize}$$\ZH\xrightarrow{(\delta_{1\ast}, \alpha_{\ast}, \gamma_{\ast})^T}\ZH/4\lrg{ i_{\overline{C}}\varrho_6}\oplus \zz \lrg{i_{\underline{M}}j_1\xi_{r'}}\oplus \zz\lrg{ i_{\underline{M}}j_2\eta}\oplus\zz\lrg{(j_1\eta)_{S}^{-}}\oplus \ZH\lrg{(2j_2)_{S}^{-}}\rightarrow \pi_9Q_1 \rightarrow0 $$\end{footnotesize}
where $\delta_{1\ast}(1)=t_1i_{\overline{C}}\varrho_6$, $\alpha_{\ast}(1)=i_{\underline{M}}j_1\xi_{r'}$ and $\gamma_{\ast}(1)=2^{s'-1}(2j_2)_{S}^{-}$.
\begin{align}
\pi_9Q_1\cong\footnotesize{\frac{\ZH/4\oplus\zz\oplus\zz\oplus\zz\oplus\ZH}{\lrg{(t_1, 1,0,0,2^{s'-1})}}\cong
\left\{
  \begin{array}{ll}
     \zz\oplus\zz\oplus\zz\oplus\z{s'+1}, & \hbox{$t_1=1$} \\
   \ZH/4\oplus \zz\oplus\zz\oplus\z{s'}, & \hbox{$t_1=2$}
  \end{array}
\right.}
. \label{pi9Q1}
\end{align}
Since $\pi_9(\CCC)\cong \zz\oplus\zz\oplus\pi_9Q_1$ induced by $\CCC\simeq C_{r}^{9,s}\vee Q_1$ for $s=r'>s'>r$, together with the short exact sequence for $\pi_9(\CCC)$ in Lemma \ref{lemma5.5}, we have
$$\pi_9Q_1= \zz\oplus\zz\oplus\zz\oplus\z{s'+1}, ~~\text{i.e.,} ~~t_1=1, ~~~~\delta_1=i_{\overline{C}}\varrho_6.$$

\textbf{(2) The cell structure of $C^{5,s'}\wedge C_{r}^{5,s}$}
\\

 Apply Lemma \ref{lemma2.4} to \textbf{Cof1} of $C^{5,s'}$ and \textbf{Cof4} of $C_{r}^{5,s}$ to get
$$(C^{5,s'}\wedge C_{r}^{5,s})^{(9)}\simeq \m{r}{3}\wedge C_{\eta}^5\vee C_{r'}^{5,s'}\vee C_{r}^9.$$
There is a cofibre sequence
\begin{align}
S^9\xlongrightarrow{(\hat{\delta}, \hat{\alpha}, \hat{\gamma})^T}\m{r}{3}\wedge C_{\eta}^5\vee C_{r'}^{9,s'}\vee C_{r}^{9}\rightarrow C^{5,s'}\wedge C_{r}^{5,s}\rightarrow S^{10} \label{CofC^sCrs}
\end{align}

where $\hat{\delta}=i_{\overline{C}}\hat{t}\varrho_6$, $\hat{\alpha}=i_{\underline{M}}\mtwo{\hat{x}\xi_{r'}}{\hat{y}\eta}$ and  $\hat{\gamma}$ is determined by $q_S\mtwo{\hat{z}\eta}{\hat{a}}$ for $\hat{x},\hat{y},\hat{z}\in\{0,1\}$ and $\hat{t}, \hat{a}\in \ZH$, i.e.,

\begin{footnotesize}$$\xymatrix{
                &         S^9 \ar[dl]_{\hat{t}\varrho_6}\ar[d]^{\hat{\delta}}     \\
S^6\ar[r]_{i_{\overline{C}}} & \m{r}{3}\wedge C_{\eta}^5~;}
~~~~~~\xymatrix{
                &         S^9 \ar[dl]_{\mtwo{\hat{x}\xi_{r'}}{\hat{y}\eta}}\ar[d]^{\hat{\alpha}}     \\
 \m{r'}{7}\vee S^8 \ar[r]_{i_{\underline{M}}} & C_{r'}^{9,s'}~;}
~~~~~~\xymatrix{
 S^9 \ar[d]_{\hat{\gamma}} \ar[dr]^{\mtwo{\hat{z}\eta}{\hat{a}}}        \\
 C_{r}^{9} \ar[r]_{q_{S}~~~}  & S^8\vee S^9~~.             }$$\end{footnotesize}
$\hat{a}=2^{s'}$ for $H_{9}(C^{5,s'}\wedge C_{r}^{5,s})\cong \z{s'}$. Cofibre sequence (\ref{CofC^sCrs}) induces
$$0\rightarrow \frac{\zz}{\lrg{\hat{x},\hat y,\hat z}} \rightarrow [C^{5,s'}\wedge C_{r}^{5,s}, S^8]\rightarrow \z{r'+1}\oplus\z{r+1}\xrightarrow{( \hat{\alpha}^{\ast}, \hat{\gamma}^{\ast})}\zz$$
where $\hat{\alpha}^{\ast}(1)=\hat x$ and $\hat{\gamma}^{\ast}(1)=\hat z$. Since $C^{5,s'}\wedge C_{r}^{5,s}$ is indecomposable for
$s=r'>s'>r$, $\hat x=1$ or $\hat y=1$, which implies $\frac{\zz}{\lrg{\hat{x},\hat y,\hat z}}=0$.
On the other hand, $$[C^{5,s'}\wedge C_{r}^{5,s}, S^8]\cong [C_{r}^{5,s}, C_{s'}^5]\cong \z{r+1}\oplus\z{r'}.$$
Thus $\hat x=1$ and $\hat z=0$. Now from the following commutative diagram
$$\xymatrix{
   S^9 \ar@{=}[d] \ar[rr]^{(\hat{\delta}, \hat{\alpha}, \hat{\gamma})^T~~~~~~~~~~~~~} & &\m{r}{3}\wedge C_{\eta}^5\vee C_{r'}^{9,s'}\vee C_{r}^{9} \ar[d]_{\Theta~(\simeq)} \ar[r] &C^{5,s'}\wedge C_{r}^{5,s} \ar@{.>}[d]^{\widetilde{\Theta}~(\simeq)} \\
   S^9 \ar[rr]^{\theta=\Theta(\hat{\delta}, \hat{\alpha}, \hat{\gamma})^T ~~~~~~~~~~~~~~~~} &&\m{r}{3}\wedge C_{\eta}^5\vee C_{r'}^{9,s'}\vee C_{r}^{9} \ar[r] & \coH{\theta}   }$$
where $\Theta=\footnotesize{\begin{tabular}{c|ccc|}
     \multicolumn{1}{c}{} &$\m{r}{3}\wedge C_{\eta}^5$&$C_{r'}^{9,s'}$&\multicolumn{1}{c}{$  C_{r}^{9} $}\\
       \cline{2-4}
    $\m{r}{3}\wedge C_{\eta}^5$&1 & 0&0\\
   $ C_{r'}^{9,s'}$&0 &$\lambda$& 0\\
    $ C_{r}^{9}$&0 & 0& 1\\
      \cline{2-4}
\end{tabular}}$ is a self-homotopy equivalence and $\lambda$ is induced by
  $$\xymatrix{
   S^8 \ar@{=}[d] \ar[rr]^{\mtwo{i\eta}{2^{s'}}~~~~~~~~} & &\m{r'}{7}\vee S^8 \ar[d]_{\mfour{1}{0}{\hat{y}q}{1}} \ar[r]^{i_{\underline{M}}} & C_{r'}^{9,s'} \ar@{.>}[d]^{\lambda}\\
   S^8 \ar[rr]_{\mtwo{i\eta}{2^{s'}}~~~~~~~~} &&  \m{r'}{7}\vee S^8 \ar[r]_{i_{\underline{M}}} & C_{r'}^{9,s'}~,}$$
we can assume $\hat y=0$. Next we are going to calculate $\hat \delta$, i.e., $\hat t$.
\\

 $\hat t=1$ for $r=1$ and $\hat t \in\{1,2\}$ for $r\geq 2$ since $\hat \delta\neq 0$.

 For the case $r\geq 2$,
  similarly as the calculation of $\pi_9Q_1$, we get
\begin{align}
\pi_9(C^{5,s'}\wedge C_{r}^{5,s})\cong \footnotesize{
\left\{
  \begin{array}{ll}
     \zz\oplus\zz\oplus\zz\oplus\z{s'+1}, & \hbox{$\hat t=1$} \\
   \ZH/4\oplus \zz\oplus\zz\oplus\z{s'}, & \hbox{$\hat t=2$}
  \end{array}
\right.}. \label{pi9C^5sCrs}
\end{align}

On the other hand, $\pi_9(C^{5,s'}\wedge C_{r}^{5,s})\cong [C_{s'}^7, C_{r}^{6,s}]$ and from \textbf{Cof3} of $C_{s'}^7$, we get the following exact sequence $$[C_{\eta}^8, C_r^{6,s}]\xrightarrow {(i_{\eta}2^{s'})^{\ast}=0} [S^6, C_r^{6,s}]\rightarrow [C_{s'}^7, C_r^{6,s}]\rightarrow[C_{\eta}^7, C_r^{6,s}]\xrightarrow {(i_{\eta}2^{s'})^{\ast}} [S^5, C_r^{6,s}].$$
There is a commutative diagram for the morphism $(i_{\eta}2^{s'})^{\ast}$
$$\xymatrix@C=0.5cm{
  [C_{\eta}^7, C_r^{6,s}] \ar[r]^{i_{\eta}^{\ast}}\ar@/^+10mm/[rr]^{(i_{\eta}2^{s'})^{\ast}} & [S^5, C_r^{6,s}] \ar[r]^{(2^{s'})^{\ast}} & [S^5, C_r^{6,s}]\\
  [C_{\eta}^7, \m{r}{4}\vee S^5]  \ar@{>>}[u]^{ (i_{\underline{M}})_{\ast}}\ar[r]^{i_{\eta}^{\ast}} & [S^5,  \m{r}{4}\vee S^5] \ar@{>>}[u]^{ (i_{\underline{M}})_{\ast}}\ar[r]^{(2^{s'})^{\ast}} & [S^5,  \m{r}{4}\vee S^5] \ar@{>>}[u]^{ (i_{\underline{M}})_{\ast}}\\
  [C_{\eta}^7, S^5]  \ar[u]^{ \mtwo{i\eta}{2^s}_{\ast}}\ar[r]^{i_{\eta}^{\ast}} & [S^5,  \m{r}{4}\vee S^5] \ar[u]^{ \mtwo{i\eta}{2^s}_{\ast}}\ar[r]^{(2^{s'})^{\ast}} & [S^5,  \m{r}{4}\vee S^5] \ar[u]^{ \mtwo{i\eta}{2^s}_{\ast}}\\
 }$$
 which induces the following commutative diagram by (\ref{map(ieta)})
 $$\xymatrix@C=0.5cm{
  [C_{\eta}^7, C_r^{6,s}] \ar[r]^{i_{\eta}^{\ast}} & [S^5, C_r^{6,s}] \ar[r]^{(2^{s'})^{\ast}} & [S^5, C_r^{6,s}]\\
   \frac{\ZH/4\oplus\ZH}{\lrg{(2, 2^s)}}\ar[u]^{ \cong}\ar[r]^{\varphi} &  \frac{\zz\oplus\ZH}{\lrg{(1, 2^s)}} \ar[u]^{ \cong}\ar[r]^{(2^{s'})^{\ast}} &  \frac{\zz\oplus\ZH}{\lrg{(1, 2^s)}} \ar[u]^{\cong}\\
  }$$
 where $\varphi(1,0)=0$ and $\varphi(0, 1)=(0, 2)$. Under the isomorphisms
 $$ \frac{\ZH/4\oplus\ZH}{\lrg{(2, 2^s)}}\cong \zz\lrg{(1, 2^{s-1})}\oplus \z{s+1}\lrg{(0,1)};~~~\frac{\zz\oplus\ZH}{\lrg{(1, 2^s)}}\cong \z{s+1}\lrg{(0,1)}$$
we have $\zz\oplus\z{s+1}\xrightarrow {(2^{s'})^{\ast}\varphi=\mtwo{0}{2^{s'+1}}}\z{s+1}$. Consequently,
$$ker(i_{\eta}2^{s'})^{\ast}\cong \zz\oplus\z{s'+1}.$$
Thus there is a short exact sequence
$$0\rightarrow \zz\oplus\zz\rightarrow [C_{s'}^7, C_r^{6,s}]\rightarrow\zz\oplus\z{s'+1}\rightarrow0$$
Together with (\ref{pi9C^5sCrs}), we have
$$\pi_9(C^{5,s'}\wedge C_{r}^{5,s})\cong \zz\oplus\zz\oplus\zz\oplus\z{s'+1},~~\text{i.e.,}~\hat t=1.$$
From the analysis above, we get $\hat\delta=\delta$, $\hat \alpha=\alpha$ and $\hat \gamma=\gamma$, i.e., $Q_1\simeq C^{5,s'}\wedge C_{r}^{5,s}$.
\end{proof}

Hence $C_{r}^{5,s}\wedge C_{r'}^{5,s'}\simeq C_{r}^{9,s}\vee C^{5,s'}\wedge C_{r}^{5,s}$ for $s=r'>s'>r$.

\end{document}